\documentclass[11pt,centertags, reqno]{amsart}

\usepackage{amsmath}
\usepackage{amsthm}
\usepackage{amssymb}
\usepackage{amstext}
\usepackage{amscd}
\usepackage{mathrsfs}
\usepackage[pdftex,a4paper]{hyperref}
\usepackage{exscale}
\usepackage{verbatim}
\usepackage[arrow, matrix, curve]{xy}
\usepackage{enumerate}
\usepackage{graphics}
\usepackage{esint}

\usepackage{fouriernc}
\usepackage[T1]{fontenc}

\numberwithin{equation}{section}

\allowdisplaybreaks

\input xy
\xyoption{all}
\theoremstyle{definition}
	\newtheorem{definition}{Definition}
	\newtheorem*{definition*}{Definition}
	
	\numberwithin{definition}{section}
\theoremstyle{plain}
	\newtheorem{lemma}[definition]{Lemma}
	
	\newtheorem{proposition}[definition]{Proposition}
	\newtheorem{theorem}[definition]{Theorem}
	\newtheorem*{theorem*}{Theorem}
	\newtheorem{corollary}[definition]{Corollary}

	\newtheorem*{claim*}{Claim}
\theoremstyle{remark}

	\newtheorem{remark}[definition]{Remark}

\newcommand{\C}{\mathbb{C}}

\newcommand{\R}{\mathbb{R}}

\newcommand{\deq}{\mathrel{\mathop:}=}

\numberwithin{equation}{section}

\renewcommand{\emph}{\textbf}

\numberwithin{equation}{section}

\begin{document}

\title[Perturbations of the Spence-Abel equation]{Perturbations of the {S}pence-{A}bel equation\\ and deformations
  of the dilogarithm function}
\author[T. Hartnick]{Tobias Hartnick}
\address{Mathematics Department, Technion, Haifa 32000, Israel}
\email{hartnick@tx.technion.ac.il}

\author[A. Ott]{Andreas Ott}
\address{Mathematisches Institut, Ruprecht-Karls-Universit\"at Heidelberg, Im Neuenheimer Feld 288, 69120 Heidelberg, Germany}
\email{aott@mathi.uni-heidelberg.de}

\begin{abstract} We analyze existence, uniqueness and regularity of solutions for perturbations of the Spence-Abel equation for the Rogers' dilogarithm. As an application we deduce a version of Hyers-Ulam stability for the Spence-Abel equation.
 Our analysis makes use of a well-known cohomological interpretation of the Spence-Abel equation and is based on our recent results on continuous bounded cohomology of ${\rm SL}_2(\R)$. 
\end{abstract}
\maketitle

\section{Introduction}
Given a system of functional equations, it is natural to investigate whether solutions to small perturbations of the equations remain close to solutions of the original system. This type of stability analysis is often referred to as the \emph{(generalized) Hyers-Ulam problem} since it was popularized by Ulam \cite[Chapter 6]{Ulam}, and the first major results (concerning the linear functional equation) were obtained by Hyers \cite{Hyers}. The investigation of generalized Hyers-Ulam stability has branched out into several different directions. In the case of of the linear functional equation on general groups it has led to the study of real-valued quasimorphisms (a notion closely related to bounded cohomology in the sense of \cite{Gromov}) and more generally perturbations of unitary representations \cite{BOT} and quasimorphisms with non-commutative target \cite{HS, FK}. There is also a wide literature concerning Hyers-Ulam stability of non-linear functional equations, whose development up to the late 90's is summarized in the monograph \cite{HIR}. For more recent developments and applications see also \cite{Jung, Kannappan}.

The present article is concerned with \emph{Hyers-Ulam stability of the Spence-Abel equation}, the functional equation of the Rogers' dilogarithm. This functional equation is substantially more involved than the equations whose Hyers-Ulam stability was considered previously. Before we state our stability result, let us briefly recall some background concerning the Rogers' dilogarithm and its functional equation.

Rogers' dilogarithm arises as a certain symmetrizations of the restriction of the \emph{dilogarithm} ${\rm Li}_2$ to $(0,1)$. On this interval the dilogarithm can be given either by an integral or a convergent power series \cite{Leibniz, Euler}:
\[
{\rm Li}_2(x) = -\int_0^x \frac{\log(1-t)}{t}dt = \sum_{n=1}^\infty \frac{x^n}{n^2}.
\]
As Zagier has pointed out \cite{Zagier}, ``\textit{[t]oday one needs no apology for devoting a paper to this function}''. This is partly due to the applications that the dilogarithm has found in conformal field theory \cite{CFT1, CFT2, CFT4, CFT3}, but also due to its ubiquity in different areas of mathematics including (but not limited to) hyperbolic geometry \cite{HyperbolicVolume}, cohomology of Lie groups \cite{Bloch, Dupont}, algebraic $K$-theory \cite{K1, Gon, K2} and modular forms \cite{ModularForms}. Numerous mysterious identities involve the dilogarithm \cite{Zagier}, most notably the \emph{$5$-term equation}, which was discovered independently by Spence \cite{Spence}, Abel \cite{Abel}, Hill \cite{Hill}, Kummer \cite{Kummer} and Schaeffer \cite{Schaeffer}. For $0<x<y<1$ it can be stated as
\begin{eqnarray*}
&&{\rm Li}_2\left(x\right) + {\rm Li}_2\left(y\right) +  {\rm Li}_2\left(\frac{1-x}{1-xy}\right) + {\rm Li}_2\left({1-xy}\right)+ {\rm Li}_2\left(\frac{1-y}{1-xy}\right)\\ &=& \frac{\pi^2}{6}-\log(x)\log(1-x) - \log(y)\log(1-y) + \log\left(\frac{1-x}{1-xy}\right)\log\left(\frac{1-y}{1-xy}\right).\end{eqnarray*}
A century after Spence and Abel, Rogers \cite{Rogers} observed that a certain symmetrization of the dilogarithm obeys a much simpler $5$-term equation. He introduced the function $L_2: (0,1) \to \R$ given by
\[L_2(x) := -\frac 1 2 \cdot \int_0^x \left( \frac{\log(t)}{1-t}+\frac{\log(1-t)}{t} \right)dt = \frac{1}{2}({\rm Li}_2(x) - {\rm Li}_2(1-x)+\zeta(2)),\]
which is nowadays referred to as \emph{Rogers' dilogarithm}, and pointed out that for this function the $5$-term equation for the Euler dilogarithm simplifies to
\begin{equation}
L_2(x) -L_2(y) - L_2\left(\frac{x}{y}\right) - L_2\left(\frac{y-1}{x-1}\right) +  L_2\left(\frac{x(y-1)}{y(x-1)}\right)=0\label{Spence-Abel}
\end{equation}
for all $0<x<y<1$. This equation is often referred to as the \emph{Spence-Abel functional equation} for $L_2$. By construction, Rogers' dilogarithm also has the \emph{reflection symmetry}
\begin{equation}\label{L2Symmetry} L_2(1-x) = \zeta(2)-L_2(x).\end{equation}

It is an amusing exercise in differentiation to prove that $L_2$ is in fact the only three-times differntiable function on the interval $(0,1)$ satisfying \eqref{Spence-Abel} and \eqref{L2Symmetry} (see \cite[Appendix A]{Dupont}). Proving uniqueness under lower regularity assumptions is a much harder problem. It was only established relatively recently by Burger and Monod \cite{Burger/On-and-around-the-bounded-cohomology-of-SL2} that any measurable solution of \eqref{Spence-Abel} and \eqref{L2Symmetry} has to agree with Rogers' dilogarithm almost everywhere. This implies in particular, that Rogers' dilogarithm is the unique continuous solution. Concerning Hyers-Ulam stability of the system \eqref{Spence-Abel} -- \eqref{L2Symmetry} we have the following main result:
\begin{theorem}[Stability of the Spence-Abel system]\label{Stability} Let $L:(0, 1) \to \R$ be a measurable function such that
\[
 \sup_{x,y \in (0,1)}\left|L(x) -L(y) - L\left(\frac{x}{y}\right) - L\left(\frac{y-1}{x-1}\right) +  L\left(\frac{x(y-1)}{y(x-1)}\right)\right| \leq \epsilon
\]
and $L(1-x) = C-L(x)$ for some $\epsilon >0$ and $C \in \R$. Then
\[
\|L-L_{2}\|_\infty \leq 11 \cdot \epsilon + 6 \cdot |C-\zeta(2)|.
\]
\end{theorem}

This stability result is a consequence of a detailed study of the perturbed Spence-Abel system
\begin{eqnarray}
\label{PerturbedSpenceAbel}L(x) -L(y) - L\left(\frac{x}{y}\right) - L\left(\frac{y-1}{x-1}\right) +  L\left(\frac{x(y-1)}{y(x-1)}\right) &=& R(x,y), \quad((x,y) \in \mathcal P_2)\\
\label{PerturbedSymmetry} L(1-x) &=& C-L(x), \quad (x \in (0,1))
\end{eqnarray}
for an arbitrary choice of constant $C \in \R$ and right hand side $R$. Here, $R$ is supposed to be a real-valued function on $\mathcal P_2 := \{g: \{(x,y)\in (0,1)^2\mid x < y\} $, and $L$ is a real-valued function on the open interval $(0,1)$. We refer to \eqref{PerturbedSpenceAbel} as the \emph{perturbed Spence-Abel equation} and to \eqref{PerturbedSymmetry} as the \emph{perturbed reflection symmetry}.

We establish the following results concerning existence, uniqueness, boundedness and regularity of solutions. Here $L^0$ denotes the space of all measurable functions modulo almost everywhere vanishing functions, $C^\omega$ denotes the space of all real-analytic functions and $\mathcal P_3 := \{(x,y,z) \in (0,1)^3\mid x<y<z\}$.

\begin{theorem}\label{IntroSummary} The perturbed Spence-Abel system \eqref{PerturbedSpenceAbel} -- \eqref{PerturbedSymmetry} has the following properties.
 \begin{description}
     \item[(Uniqueness)] For every $R \in L^0(\mathcal P_2)$ and $C\in \R$ the system \eqref{PerturbedSpenceAbel} -- \eqref{PerturbedSymmetry} has at most one solution $L^{(R, C)} \in L^0((0,1))$.
    \item[(Existence)] Given $R \in L^0(\mathcal P_2)$ and $C\in \R$, the system \eqref{PerturbedSpenceAbel} -- \eqref{PerturbedSymmetry} admits a solution if and only if $R$ satisfies the following equations for almost all $(x,y,z) \in \mathcal P_3$:
\begin{eqnarray}
 R(x,y) - R(x,z) +R(y,z) - R\left(\frac{x}{z}, \frac{y}{z}\right) - R\left(\frac{z-1}{x-1}, \frac{z-1}{y-1} \right) +R\left(\frac{x(z-1)}{z(x-1)}, \frac{y(z-1)}{z(y-1)}\right) &=& 0, \label{6termlongIntro}\\
R\left(1-y, \frac{1-y}{1-x}\right) - R(x, y) &=&0.\label{RSymmetry}
\end{eqnarray}
\item[(Boundedness)] Given $R \in L^0(\mathcal P_2)$ satisfying \eqref{6termlongIntro} and $C \in \R$, the function $L^{(R, C)} $ is essentially bounded if and only if $R$ is essentially bounded.
\item[(Regularity)] Assume that $R \in L^\infty(\mathcal P_2)$ is essentially bounded and satisfies \eqref{6termlongIntro} and let $C \in \R$ and $k \in \mathbb N \cup\{0, \infty, \omega\}$. Then $L^{(R, C)}$ is of class $C^k$ if and only if $R$ is of class $C^k$.
\item[(Continuity)]  Assume that $R_1, R_2 \in L^\infty(\mathcal P_2)$ are essentially bounded and satisfy \eqref{6termlongIntro} and let $C_1, C_2 \in \R$. Then
\[
\|L^{(R_1, C_1)}-L^{(R_2, C_2)}\|_\infty \leq 11 \cdot \|R_1-R_2\|_\infty + 6 \cdot |C_1-C_2|.
\]
\end{description}
\end{theorem}
Equation \eqref{6termlongIntro} is called the \emph{$6$-term equation} for $R$. Both the Spence-Abel equation and the 6-term equation admit a cohomological interpretation. We learned about the cohomological interpretation of the Spence-Abel equation from \cite{Burger/On-and-around-the-bounded-cohomology-of-SL2}, and although the idea seems to have been around for some time, this is the only explicit reference we know. Developing this idea further, we provide in the present article a  dictionary between functional equations such as \eqref{PerturbedSpenceAbel} -- \eqref{PerturbedSymmetry} and corresponding cohomological statements. This dictionary is really at the heart of our approach, and although the statements of Theorem \ref{IntroSummary} do not involve any cohomology, our proofs are entirely cohomological.

It follows from our dictionary that the existence and uniqueness statements in Theorem \ref{IntroSummary} can be reinterpreted as vanishing results in measurable cohomology in degrees $3$ and $4$ for the ${\rm PSL}_2(\R)$-action on the circle. These vanishing results can be obtained by identifying the measurable cohomology of the boundary action of ${\rm PSL}_2(\R)$ with the measurable group cohomology of ${\rm PSL}_2(\R)$. To put this last result into perspective, let us mention that the measurable cohomology of any semisimple Lie group is expected to be realizable over the Furstenberg boundary. For ${\rm PSL}_2(\C)$ this conjecture was established by Bloch \cite{Bloch}, but there are major technical difficulties in generalizing his argument to wider classes of groups. Recently, a breakthrough in this direction was achieved by Pieters \cite{Pieters}, who established the conjecture for all real-hyperbolic groups. This includes in particular the case of ${\rm PSL}_2(\R)$ which we need here (and which is by far the simplest case).

Our contribution lies mainly in establishing boundedness of the solution for a given bounded right-hand side. In cohomological terms, this amounts to proving vanishing theorems for the continuous bounded cohomology of ${\rm PSL}_2(\R)$ in degree $3$ and $4$. (Here, (continuous) bounded cohomology is understood in the sense of \cite{Gromov, Ivanov, BuMo, MonodBook}.) In degree $3$, the relevant vanishing theorem is due to Burger and Monod \cite{Burger/On-and-around-the-bounded-cohomology-of-SL2}, whereas the vanishing theorem in degree $4$ was established in \cite{HO1}. The boundedness part of Theorem \ref{IntroSummary} is essentially a straightforward application of these two theorems. 

In order to establish the final two parts of the theorem we capitalize on the fact that the proof of the vanishing theorem for the $4$th continuous bounded cohomology of ${\rm PSL}_2(\R)$ provided in \cite{HO1} is constructive. Namely, in \cite{HO1} we construct for every $4$-cocycle an explicit primitive by integration. This means that for a given bounded measurable right-hand side $R$ we solve the system \eqref{PerturbedSpenceAbel} --\eqref{PerturbedSymmetry} explicitly. The result is as follows:
\begin{theorem} Assume that  $R \in L^\infty(\mathcal P_2)$ satisfies \eqref{6termlongIntro} -- \eqref{RSymmetry}, and let $C \in \R$. Then the unique solution $L^{(R, C)}$ of \eqref{PerturbedSpenceAbel} -- \eqref{PerturbedSymmetry} is of the form
\begin{eqnarray*}
L^{(R,C)}(x) &=& C/2 - \frac{1}{2\pi}\int_0^{2\pi} c^{(R,C)}\left(e^{i\psi}, 1, -1, \mathcal C(x), -i\right) d\psi\\
&& - \int_0^{\frac{-x-1}{2x}} F_{(R,C)}^{\flat}\left(2\cdot {\rm arccot}\left(-t+\frac{1-x}{2x} \right),2\cdot {\rm arccot}\left(-t-\frac{1-x}{2x} \right)\right)dt\\
&& + \int_0^{\frac{x+1}{2}} F_{(R,C)}^{\flat}\left(2\cdot {\rm arccot}\left(-t+\frac{1-x}{2} \right),2\cdot {\rm arccot}\left(-t-\frac{1-x}{2} \right)\right)dt\\
&&-  \int_0^{\frac{1}{2}}  F_{(R,C)}^{\flat}\left(2\cdot {\rm arccot}\left(-t+\frac{1}{2} \right), 2\cdot {\rm arccot}\left(-t-\frac{1}{2} \right)\right)dt\\
&& + \int_0^{\frac x 2} F_{(R,C)}^{\flat}\left(2\cdot {\rm arccot}\left(-t+\frac x 2 \right), 2\cdot {\rm arccot}\left(-t-\frac x 2 \right)\right)dt,
\end{eqnarray*}
where $c^{(R,C)} \in L^\infty((S^1)^5)$ and $F^\flat_{(R, C)} \in L^\infty((0, 2\pi)^2)$ are certain explicit functions constructed from the function $R - C/2$ as explained in Subsection \ref{SecExplicit}.
\end{theorem}

Both regularity of the solution and continuous dependence on initial conditions are consequences of these explicit formulas. In particular, applying our integration formulas to the original Spence-Abel system we obtain a new formula for Rogers' dilogarithm. 
\begin{theorem} The Rogers' dilogarithm $L_2$ is given by the formula
\begin{eqnarray*}
L_2(x) &=& \frac{\zeta(2)}{2} -  \frac{\zeta(2)}{2\pi}\left( \arctan\left(\frac{2x}{1-x^2}\right) + \frac{\pi}{2} \right)\\
&& - \int_0^{\frac{-x-1}{2x}} F_{c}^{\flat}\left(2\cdot {\rm arccot}\left(-t+\frac{1-x}{2x} \right),2\cdot {\rm arccot}\left(-t-\frac{1-x}{2x} \right)\right)dt\\
&& + \int_0^{\frac{x+1}{2}} F_{c}^{\flat}\left(2\cdot {\rm arccot}\left(-t+\frac{1-x}{2} \right),2\cdot {\rm arccot}\left(-t-\frac{1-x}{2} \right)\right)dt\\
&&-  \int_0^{\frac{1}{2}}  F_{c}^{\flat}\left(2\cdot {\rm arccot}\left(-t+\frac{1}{2} \right), 2\cdot {\rm arccot}\left(-t-\frac{1}{2} \right)\right)dt\\
&& + \int_0^{\frac x 2} F_{c}^{\flat}\left(2\cdot {\rm arccot}\left(-t+\frac x 2 \right), 2\cdot {\rm arccot}\left(-t-\frac x 2 \right)\right)dt,
\end{eqnarray*}
where 
\begin{eqnarray*}
F_c^\flat(\varphi_1, \varphi_2)&=&  \frac{\zeta(2)}{4\pi^2} \cdot \left( (3\varphi_1-2\pi)(\cos(\varphi_2)-1) - (3\varphi_2-4\pi)(\cos(\varphi_1)-1)\right) \\
&& +  \frac{3\zeta(2)}{8\pi^2} \cdot \left(\sin(\varphi_1) \cdot \log\left(\frac{\sin((\varphi_2-\varphi_1)/2)}{\sin(\varphi_1/2)}\right)- \sin(\varphi_2) \cdot \log\left(\frac{\sin((\varphi_2-\varphi_1)/2)}{\sin(\varphi_2/2)}\right)\right).
\end{eqnarray*}
\end{theorem}

This article is organized as follows: In Section \ref{SecDictionary} we provide a dictionary between cohomologies of certain complexes of ${\rm PSL}_2(\R)$-invariant functions on tori and solutions of certain functional equations. Existence and uniqueness for the perturbed Spence-Abel system are then deduced in Section \ref{SecExUn} from the cohomological vanishing results mentioned above. Section \ref{SectionExplicit} is concerned with explicit integral formulas for solutions of the perturbed Spence-Abel system. In particular, we derive the desired regularity and continuity results. Finally, in Section \ref{SecRogers} we derive the above formula for Rogers' dilogarithm.

\textbf{Acknowledgement:} The authors thank Nicolas Monod and Hester Pieters for their patient explanations concerning \cite{Burger/On-and-around-the-bounded-cohomology-of-SL2} and \cite{Pieters} respectively. They are indepted to Marc Burger, Simeon Reich and Themistocles Rassias for pointing out various references concerning Hyers-Ulam stability. T.H. was partially supported by ISF grant 2021372. A.O. was supported by the European Research Council under ERC-Consolidator grant 614733 ``Deformation Spaces of Geometric Structures".

\section{Cohomological formulation of the $5$- and $6$-term equations}\label{SecDictionary}

\subsection{The boundary action of ${\rm PU}(1,1)$ and the cross ratio}
We consider the action of the group ${\rm PSL}_2(\C)$ on the extended complex plane $\widehat{\C} = \C \cup \{\infty\}$ by M\"obius transformations
\[
 \left(\begin{matrix} a&b\\ c&d \end{matrix}\right).z = \frac{az+b}{cz+d}.
\]
Given a space $X$ we denote by $X^{(n)} \subset X^n$ the subset of $n$-tuples of pairwise distinct points.
\begin{definition}
The \emph{cross ratio} $[-:-:-:-]: \widehat{\C}^{(4)} \to \C$ is the function given by
\[
[z_1:z_2:z_3:z_4] := \frac{(z_1-z_3)(z_2-z_4)}{(z_2-z_3)(z_1-z_4)} \quad ((z_1, z_2, z_3, z_4) \in \widehat{\C}^{(4)}).
\]
\end{definition}
The cross ratio as defined above is the unique ${\rm PSL}_2(\C)$-invariant function on $\widehat{\C}^{(4)}$ subject to the normalisation
 \begin{equation}\label{CRNormalization}
z = [z:1:0:\infty] =[\infty:0:1:z].
 \end{equation}
 We warn the reader that there exist several other normalization conventions for the cross ratio in the literature. We will always stick to the normalization above. With this normalization we then have
 the following cocycle identities for all $z \in \widehat{\C} \setminus \{z_1, \dots, z_4\}$:
\begin{equation}\label{cocycle}
[z_1:z_2:z_3:z_4] = [z_1:z:z_3:z_4][z:z_2:z_3:z_4]=[z_1:z_2:z_3:z][z_1:z_2:z:z_4].
\end{equation}
We will be interested in the subgroup $G := {\rm PU}(1,1)<{\rm PSL}_2(\C)$, whose elements are represented by matrices of the form
\[
g_{a,b} := \left(\begin{matrix} a&b\\ \bar b&\bar a \end{matrix}\right),
\]
with $a,b \in \C$ subject to the condition $|a|^2-|b|^2=1$. The action of $G$ on $\widehat{\C}$ preserves the unit disc $\mathbb D \subset \widehat{\C}$ and its boundary, the circle $S^1$. We refer to the action of $G$ on $S^1$ as the \emph{boundary action} of $G$.  

\subsection{Orbits of cyclically oriented points and cross ratio coordinates}

The boundary action of $G = {\rm PU}(1,1)$ on $S^1$ induces a diagonal action of $G$ on $(S^1)^k$ for every $k \in \mathbb N$, which commutes with the natural action of the symmetric group $\mathfrak S_k$ by permuting the coordinates. Both actions preserve the open subset $(S^1)^{(k)} \subset (S^1)^k$, and we are interested in the $G\times \mathfrak S_k$ orbits in $(S^1)^{(k)}$.

The $G$-action on $(S^1)^{(3)}$ is free and has exactly two orbits $(S^1)^{(3, \pm)}$ given by positively and negatively oriented triples respectively. Indeed, this follows from the fact that ${\rm PGL}_2(\R)$ acts sharply $3$-transitively on $\widehat{\R}^{(3)}$ and that the index $2$ subgroup ${\rm PSL_2}(\R)$ preserves the subset $\widehat{\R}^{(3,+)}$. Even permutations in $\mathfrak S_3$ preserve these two orbits, whereas odd permutations exchange them. 
 \begin{definition}
Let $k \geq 3$. A $k$-tuple $(z_1, \dots, z_k) \in (S^1)^{(k)}$ is called \emph{cyclically oriented} if there exist $(\theta_1, \dots, \theta_k) \in \R$ such that $z_j = e^{i\theta_j}$ for all $j = 1, \dots, k$ and
\[\theta_1<\theta_2< \dots < \theta_k < \theta_1 + 2\pi.\]
We denote by $(S^1)^{(k,+)} \subset (S^1)^{(k)}$ the set of cyclically oriented $k$-tuples.
\end{definition}
 For $k = 3$, cyclically oriented triples are the same as positively oriented triples, hence the notation is compatible with the previous one. Note that every ordered subtuple of a cyclically oriented $k$-tuple is cyclically oriented, and that a $k$-tuple is cyclically oriented if and only if every ordered $4$-subtuple is cyclically oriented. We can use these facts to characterize cyclically oriented $k$-tuples by means of cross ratios.
\begin{lemma}\label{CRParameterSpace} Let $(z_1, \dots, z_k) \in (S^1)^{k}$. Then $(z_1, \dots, z_k)$ is cyclically oriented if and only if $(z_1, z_2, z_3)$ is cyclically oriented and
\[
0 < [z_2:z_3:z_1:z_4] < [z_2:z_3:z_1:z_5] < [z_2:z_3:z_1:z_6] < \dots < [z_2:z_3:z_1:z_k] < 1.
\]
\end{lemma}
\begin{proof} Consider the function $f_0: S^1 \setminus\{z_1, z_2, z_3\} \to \R$ given by $f_0(z) = [z_2:z_3:z_1:z]$. It follows from the explicit formula for the cross ratio that $f_0$ extends to a homeomorphism $f: S^1 \to \widehat{\mathbb{R}}$ such that $f(z_1) = 1$, $f(z_2) = \infty$ and $f(z_3) = 0$. In particular, $z$ lies between $z_3$ and $z_1$ if and only if $f(z) \in (0,1)$ and $z_j$ lies before $z_{j+1}$ if and only if $f(z_j) < f(z_{j+1})$.
\end{proof}
The lemma motivates the following definition.
\begin{definition} Let $k \geq 4$ and $(z_1, \dots, z_k) \in (S^1)^{(k,+)}$. Then the numbers $\lambda_1, \dots, \lambda_{k-3} \in (0,1)$ given by
\[
\lambda_j := [z_2:z_3:z_1:z_{j+3}]
\]
are called the \emph{cross ratio coordinates} of $(z_1, \dots, z_k)$.
\end{definition}
If we abbreviate by $\mathcal P_n$ the parameter space
\[
\mathcal P_n := \{(x_1, \dots, x_n) \in (0,1)^n\mid x_1 < x_2 < \dots < x_n\} \subset (0,1)^n, 
\]
then by Lemma \ref{CRParameterSpace} cross ratio coordinates define a map
\[
\Lambda: (S^1)^{(k,+)} \to \mathcal P_{k-3}, \quad (z_1, \dots, z_k) \mapsto (\lambda_1, \dots, \lambda_{k-3}).
\]
To see that this map is onto we recall that the \emph{Cayley transform} $\mathcal C: \widehat{\C} \to \widehat{\C}$ is the M\"obius transformation given by \[\mathcal C(z) = \frac{z-i}{z+i}, \quad\mathcal C^{-1}(w) = i\frac{1+w}{1-w}.\]
The Cayley transform restricts to a bijection $\mathcal C: \widehat{\R} \to S^1$ and intertwines the ${\rm PSL}_2(\R)$-action on $\widehat{\R}$ and the $G$-action on $S^1$. Since $\mathcal C(\infty, 0,1) = (1, -1, -i)$ we can reformulate the identity \eqref{CRNormalization} as
 \[
\mathcal C^{-1}(z) = [\infty:0:1:\mathcal C^{-1}(z)] = [\mathcal C(\infty): \mathcal C(0): \mathcal C(1): \mathcal C(\mathcal C^{-1}(z))] =  [1:-1:-i:z] \in \R \setminus\{0,1\}.
\]
\begin{lemma}\label{LemmaRestriction} Let $(\lambda_1, \dots, \lambda_{k-3}) \in \mathcal P_{k-3}$. Then
\[
(z_1, \dots, z_k) := (-i,1,-1,\mathcal C(\lambda_1), \dots, \mathcal C(\lambda_{k-3}))\in (S^1)^{(k,+)}
\]
and $(\lambda_1, \dots, \lambda_{k-3}) = \Lambda(z_1, \dots, z_k)$. In particular, $\Lambda$ is onto.
\end{lemma}
\begin{proof} By definition, $(z_1, z_2, z_3) = (-i, 1, -1)$ is cyclically oriented and for every $j = 4, \dots, k$ we have 
\[
[z_2:z_3:z_1: z_j] = [1:-1:-i:z_j] = \mathcal C^{-1}(z_j) = \lambda_{j-3}.
\]
Thus the lemma follows from Lemma \ref{CRParameterSpace}.
\end{proof}
Since $G$ acts sharply $3$-transitively on $(S^1)^{(3,+)}$, it follows from Lemma \ref{CRParameterSpace} that $\Lambda$ induces a bijection $G\backslash (S^1)^k \to \mathcal P_{k-3}$. In particular, $G$-invariant functions on $(S^1)^{(k,+)}$ correspond bijectively to functions on $\mathcal P_{k-3}$ via $\Lambda$.
\begin{lemma}\label{AlternatingCRC}
If $(z_1, \dots, z_k) \in (S^1)^{(k,+)}$ has cross ratio coordinates $(\lambda_1, \dots, \lambda_{k-3})$, then $(z_k, z_1, \dots, z_{k-1})$ has cross ratio coordinates $(\widetilde{\lambda_1}, \dots, \widetilde{\lambda_{k-3}})$ given by $\widetilde{\lambda_1} = 1-\lambda_{k-3}$ and
\[
\widetilde{\lambda_j}  =  \frac{1-\lambda_{k-3}}{1-\lambda_{j-1}} \quad (j=2, \dots, k-3).
\]
\end{lemma}
\begin{proof} By the standard symmetries of the cross ratio we have
\[
\widetilde{\lambda_1} = [z_1:z_2:z_k:z_3] = 1-[z_2:z_3:z_1:z_k] = 1-\lambda_{k-3},
\]
and for $j = 2, \dots, k-3$ we deduce from the cocycle identity \eqref{cocycle} that 
\begin{eqnarray*}
\widetilde{\lambda_j} &=& [z_1:z_2:z_k:z_{j+2}] =  [z_1:z_2:z_k:z_{3}][z_1:z_2:z_3:z_{j+2}]\\ &=& (1-[z_2:z_3:z_1:z_k])\cdot \frac{1}{1-[z_2:z_3:z_1:z_{j+2}]}= \frac{1-\lambda_{k-3}}{1-\lambda_{j-1}}.\qedhere
\end{eqnarray*}
\end{proof}

\subsection{Alternating functions in cross ratio coordinates} Our next goal is to single out the alternating $G$-invariant functions on $(S^1)^{(k)}$ in terms of their cross ratio coordinates. Here a function $f: (S^1)^{(k)} \to \R$ is called \emph{alternating} provided 
\[
f(z_{\sigma(1)}, \dots, z_{\sigma(k)}) = (-1)^\sigma f(z_1, \dots, z_k) \quad(\sigma \in \mathfrak S_k, (z_1, \dots, z_k) \in (S^1)^{(k)}),
\]
and we denote by $\mathcal F_{\rm alt}((S^1)^{(k)})$ the space of all such functions. 

We observe that every $(x_1, \dots, x_k) \in (S^1)^{(k)}$ can be permuted into a cyclically oriented $k$-tuple. If we denote by $C_k \subset \mathfrak S_k$ the cyclic group generated by the cycle $\sigma_k := (1 \, 2 \dots k)$, then this induces a bijection
\begin{equation}\label{Permutation+}
(S^1)^{(k)}/\mathfrak S_k \cong (S^1)^{(k,+)}/C_k.
\end{equation}
\begin{proposition}\label{Parameters} For $k \geq 4$ the following spaces are in bijection:
\begin{enumerate}[(1)]
\item the space $\mathcal F_{\rm alt}((S^1)^{(k)})^G$ of $G$-invariant alternating functions $c: (S^1)^{(k)} \to \R$;
\item the space $\mathcal F_{\rm alt}((S^1)^{(k,+)})^G$  of $G$-invariant functions $c_0: (S^1)^{(k, +)} \to \R$ satsfying
\begin{equation}\label{AlternatingSemiReduced}
c_0(z_k, z_1, \dots, z_{k-1}) = (-1)^{k+1} c_0(z_1, \dots, z_k).
\end{equation}
\item the space $\mathcal F_{\rm alt}(\mathcal P_{k-3})$ of functions $f: \mathcal P_{k-3} \to \R$ satisfying
\begin{equation}\label{AlternatingFinal}
f(\lambda_1, \dots, \lambda_{k-3}) = (-1)^{k+1}\cdot f\left(1-\lambda_{k-3}, \frac{1-\lambda_{k-3}}{1-\lambda_1}, \dots, \frac{1-\lambda_{k-3}}{1-\lambda_{k-4}}\right).
\end{equation}
\end{enumerate}
The bijection $(1) \leftrightarrow (2)$ is given by restriction, while the bijection $(2) \leftrightarrow (3)$ is induced by the cross ratio coordinates $\Lambda$.
\end{proposition}
\begin{proof} By \eqref{Permutation+}, $\mathfrak S_k$-equivariant functions on $(S^1)^{(k)}$ correspond via restriction to $C_k$-equivariant functions on $(S^1)^{(k,+)}$. Now $C_k$-equivariance is equivalent to \eqref{AlternatingSemiReduced} since $C_k$ is generated by the $k$-cycle $\sigma_k$. This shows that restriction defines a $G$-equivariant bijection $\mathcal F_{\rm alt}((S^1)^{(k)}) \to \mathcal F_{\rm alt}((S^1)^{(k,+)})$.

Moreover, we know from the previous subsection that cross ratio coordinates induce a bijection $\mathcal F((S^1)^{(k,+)})^G \to \mathcal F(P_n)$, and by Lemma \ref{AlternatingCRC} this bijection translates the alternating condition \eqref{AlternatingSemiReduced} into \eqref{AlternatingFinal}.
\end{proof}
Proposition \ref{Parameters} and Lemma \ref{LemmaRestriction} yield for every $k \geq 4$ mutually inverse bijections given by
\begin{eqnarray*}
{\rm ext}_k: \mathcal F_{\rm alt}(\mathcal P_{k-3}) \to \mathcal F_{\rm alt}((S^1)^{(k)})^G, &&{\rm ext}_k(f)(z_1, \dots, z_k) = f(\Lambda(z_1, \dots, z_k))\quad ((z_1, \dots, z_k) \in (S^1)^{(k, +)}),\\
{\rm res}_k: \mathcal F_{\rm alt}((S^1)^{(k)})^G \to \mathcal F_{\rm alt}(\mathcal P_{k-3}), && {\rm res}_k(c)(\lambda_1, \dots, \lambda_{k-3})= c(-i, 1, -1, \mathcal C(\lambda_1), \dots, \mathcal C(\lambda_{k-3})),
\end{eqnarray*}
which we refer to as \emph{extension} and \emph{restriction maps} respectively. When $k$ is clear from the context we simply write ${\rm ext}$ and ${\rm res}$ for ${\rm ext}_k$ and ${\rm res}_k$.

We will be specifically interested in the space $\mathcal F_{\rm alt}(\mathcal P_k)$ for $k\leq 3$ and the corresponding extension and restriction maps. Explicitly, these are given by
\begin{eqnarray}
\label{P1}\mathcal F_{\rm alt}(\mathcal P_1) &=& \{f:(0,1) \to \R\mid f(x) = -f(1-x)\},\\
\label{P2}\mathcal F_{\rm alt}(\mathcal P_2) &=& \left\{g: \{(x,y)\in (0,1)^2\mid x < y\} \to \R\mid g(x, y) = g\left(1-y, \frac{1-y}{1-x}\right)\right\},\\
\mathcal F_{\rm alt}(\mathcal P_3) &=& \left\{h: \{(x,y,z)\in (0,1)^3\mid x < y < z\} \to \R\mid h(x, y, z) = -h\left(1-z, \frac{1-z}{1-x}, \frac{1-z}{1-y}\right)\right\}.
\end{eqnarray}
Note in particular that $L:(0,1) \to \R$ satisfies the perturbed reflection symmetry
\begin{equation}\label{PerturbedSymmetry0}
L(x) = -L(1-x)
\end{equation}
with constant $C=0$ if and only if $L\in \mathcal F_{\rm alt}(\mathcal P_1)$.

\subsection{The homogeneous differential in cross ratio coordinates} Denote by $\mathcal F((S^1)^{(k)})$ the space of real-valued functions on $(S^1)^{(k)}$. Recall that the \emph{homogeneous differential} is given by
\[
\delta^n : \mathcal{F}((S^{1})^{(n+1)}) \to \mathcal{F}((S^{1})^{(n+2)}), \quad (\delta^n c)(z_{0}, \ldots, z_{n+1}) = \sum_{j=0}^{n+1} (-1)^{j} \, c(z_{0}, \ldots, \widehat{z_{j}}, \ldots, z_{n+1}).
\]
It is $G$-invariant, satisfies $\delta^n \circ \delta^{n-1} = 0$ and intertwines the corresponding subspaces of alternating functions. For $k \geq 4$ we define the \emph{reduced differential}
\[
\tau^k: \mathcal F_{\rm alt}(\mathcal P_{k-2}) \to  \mathcal F_{\rm alt}(\mathcal P_{k-1}), \quad f \mapsto {\rm res}_{k+2} \circ \delta^k \circ {\rm ext}_{k+1}(f)
\]
and obtain the following commuting diagram in which vertical arrows are isomorphisms:
\begin{equation}\label{DefTau}
\begin{xy}\xymatrix{
\dots \ar[rr]^0 &&  \mathcal{F}_{\rm alt}((S^{1})^{(4)})^G  \ar[rr]^{\delta^3} \ar[d]^{{\rm res}_4}&&   \mathcal{F}_{\rm alt}((S^{1})^{(5)})^G \ar[rr]^{\delta^4}  \ar[d]^{{\rm res}_5}&&   \mathcal{F}_{\rm alt}((S^{1})^{(6)})^G \ar[rr]^{\delta^5} \ar[d]^{{\rm res}_6}&&\dots\\ &&\mathcal F_{\rm alt}(\mathcal P_1) \ar[rr]^{\tau^3}&&\mathcal F_{\rm alt}(\mathcal P_2) \ar[rr]^{\tau^4}&&\mathcal F_{\rm alt}(\mathcal P_3)  \ar[rr]^{\tau^5}&& \dots
}\end{xy}\end{equation}
\begin{remark} Instead of considering arbitrary real-valued functions, we could as well consider bounded functions, or functions of a given regularity (measurable, continuous, $C^k$, smooth). In each of these cases we obtain a similar diagram as in \eqref{DefTau}. In the case of measurable (bounded) functions we can also pass to the respective quotients by identifying functions which agree almost everywhere. The following complex will be of particular importance for us later on.
\begin{equation}\label{DefTauLInftyDownstairs}
\begin{xy}\xymatrix{
\dots \ar[rr]^0 &&   L^\infty_{\rm alt}((S^{1})^{(4)})^G  \ar[rr]^{\delta^3} \ar[d]^{{\rm res}_4}&&    L^\infty_{\rm alt}((S^{1})^{(5)})^G \ar[rr]^{\delta^4}  \ar[d]^{{\rm res}_5}&&    L^\infty_{\rm alt}((S^{1})^{(6)})^G \ar[rr]^{\delta^5} \ar[d]^{{\rm res}_6}&&\dots\\ && L^\infty_{\rm alt}(\mathcal P_1) \ar[rr]^{\tau^3}&& L^\infty_{\rm alt}(\mathcal P_2) \ar[rr]^{\tau^4}&& L^\infty_{\rm alt}(\mathcal P_3)  \ar[rr]^{\tau^5}&& \dots
}\end{xy}\end{equation}
\end{remark}
\begin{proposition} \label{6termexplicit} For $k = 3,4$ the reduced differential $\tau^k$ is given as follows.
\begin{enumerate}
\item Let $f \in \mathcal F_{\rm alt}(\mathcal P_1)$. Then
\[
\tau^3(f)(x,y) = f(x) -f(y) - f\left(\frac{x}{y}\right) - f\left(\frac{y-1}{x-1}\right) +  f\left(\frac{x(y-1)}{y(x-1)}\right).
\]
\item Let $g \in \mathcal F_{\rm alt}(\mathcal P_2)$. Then
\[
\tau^4(g)(x,y,z) = -g(x,y) + g(x,z) - g(y,z) + g\left(\frac{x}{z}, \frac{y}{z}\right) + g\left(\frac{z-1}{x-1}, \frac{z-1}{y-1} \right) -g\left(\frac{x(z-1)}{z(x-1)}, \frac{y(z-1)}{z(y-1)}\right).
\]
\end{enumerate}
\end{proposition}
\begin{proof} (i) Since for cyclically oriented $(z_0, \dots, z_4)$ we have
\begin{eqnarray*}
 (\delta \circ {\rm ext})(f)(z_0, \dots, z_4) &=& {\rm ext}(f)(z_1, z_2, z_3, z_4) - {\rm ext}(f)(z_0, z_2, z_3, z_4)+{\rm ext}(f)(z_0, z_1, z_3, z_4)\\
 &&-{\rm ext}(f)(z_0, z_1, z_2, z_4)+{\rm ext}(f)(z_0, z_1, z_2, z_3)\\
 &=& f([z_2:z_3:z_1:z_4])- f([z_2:z_3:z_0:z_4])+ f([z_1:z_3:z_0:z_4])\\
 &&- f([z_1:z_2:z_0:z_4])+ f([z_1:z_2:z_0:z_3]).
\end{eqnarray*}
Using \eqref{P1} we deduce that
\begin{eqnarray*}
\tau^3(f)(x,y) &=& (\delta \circ {\rm ext})(f)(-i, 1, -1, \mathcal C(x), \mathcal C(y))\\
&=& f([-1:\mathcal C(x): 1: \mathcal C(y)])-f([-1:\mathcal C(x): -i: \mathcal C(y)])+f([1:\mathcal C(x): -i: \mathcal C(y)])\\
&&-f([1:-1:-i: \mathcal C(y)])+f([1:-1:-i: \mathcal C(x)])\\
&=& f([0:x:\infty:y])-f([0:x:1:y]) + f([\infty:x:1:y]) - f(y) + f(x)\\
&=& f\left(1- \frac x y\right) - f\left(1-\frac{x(y-1)}{y(x-1)}\right) + f\left(1-\frac{y-1}{x-1}\right)-f(y) + f(x)\\
&=& -f\left(\frac{x}{y}\right) +f\left(\frac{x(y-1)}{y(x-1)}\right) - f\left(\frac{y-1}{x-1}\right) - f(y) +f(x).
\end{eqnarray*}
(ii) For cyclically oriented $(z_0, \dots, z_5)$ we have
\begin{eqnarray*}
 (\delta \circ {\rm ext})(g)(z_0, \dots, z_5) &=&{\rm ext}(g)(z_1, z_2, z_3, z_4, z_5) - {\rm ext}(g)(z_0, z_2, z_3, z_4, z_5)+{\rm ext}(g)(z_0, z_1, z_3, z_4, z_5)\\
 &&-{\rm ext}(g)(z_0, z_1, z_2, z_4, z_5)+{\rm ext}(g)(z_0, z_1, z_2, z_3, z_5) - {\rm ext}(g)(z_0, z_1, z_2, z_3, z_4)\\
 &=& g([z_2:z_3:z_1:z_4], [z_2:z_3:z_1:z_5])-g([z_2:z_3:z_0:z_4], [z_2:z_3:z_0:z_5])\\
 &&+g([z_1:z_3:z_0:z_4], [z_1:z_3:z_0:z_5])-g([z_1:z_2:z_0:z_4], [z_1:z_2:z_0:z_5])\\
 &&+g([z_1:z_2:z_0:z_3], [z_1:z_2:z_0:z_5])-g([z_1:z_2:z_0:z_3], [z_1:z_2:z_0:z_4]),
 \end{eqnarray*}
 and thus, using \eqref{P2},
 \begin{eqnarray*}
 \tau^4(g)(x,y, z) &=& (\delta \circ {\rm ext})(g)(-i, 1, -1, \mathcal C(x), \mathcal C(y), \mathcal C(z))\\
 &=& g([-1:\mathcal C(x):1:\mathcal C(y)], [-1:\mathcal C(x):1:\mathcal C(z)])-g([-1:\mathcal C(x):-i:\mathcal C(y)], [-1:\mathcal C(x):-i:\mathcal C(z)])\\
 &&+g([1:\mathcal C(x):-i:\mathcal C(y)], [1:\mathcal C(x):-i:\mathcal C(z)])-g([1:-1:-i:\mathcal C(y)], [1:-1:-i:\mathcal C(z)])\\
 &&+g([1:-1:-i:\mathcal C(x)], [1:-1:-i:\mathcal C(z)])-g([1:-1:-i:\mathcal C(x)], [1:-1:-i:\mathcal C(y)])\\
 &=& g\left(1- \frac x y, 1- \frac x z\right)- g\left(1-\frac{x(y-1)}{y(x-1)},1-\frac{x(z-1)}{z(x-1)} \right) + g\left(1-\frac{y-1}{x-1},1-\frac{z-1}{x-1} \right)\\
 &&-g(y,z)+g(x,z) - g(x, y)\\
 &=& g\left(\frac{x}{z}, \frac{y}{z}\right) -g\left(\frac{x(z-1)}{z(x-1)}, \frac{y(z-1)}{z(y-1)}\right) + g\left(\frac{z-1}{x-1}, \frac{z-1}{y-1} \right)-g(y,z)+g(x,z) - g(x, y).
 \end{eqnarray*}
\end{proof}
Comparing these formulas to the formulas appearing in the introduction we deduce:
\begin{corollary}\label{CorDictionary}
The perturbed Spence-Abel equation \eqref{PerturbedSpenceAbel} and the $6$-term equation \eqref{6termlongIntro} can respectively be written as
\[\tau^3L = R\quad \text{and}\quad \tau^4 R = 0.\]
and are thus respectively equivalent to the cohomological equations
\[
\delta^3{\rm ext}_4(L) = {\rm ext}_5(R)\quad \text{and} \quad \delta^4{\rm ext}_5(R) = 0.
\]\qed
\end{corollary}
As a first consequence we see that \eqref{6termlongIntro} is a necessary condition for the solvability of \eqref{PerturbedSpenceAbel}.

\section{Existence, uniqueness and boundedness}\label{SecExUn}

\subsection{Measurable group cohomology of ${\rm PU}(1,1)$, existence and uniqueness}

The following theorem appears explicitly first in \cite{Pieters}, where it is generalized to arbitrary real-hyperbolic groups. Note that the proof for ${\rm PU}(1,1)$ does not require the more sophisticated tools developed in \cite{Pieters} to deal with cohomology with coefficients in non-Fr\'{e}chet modules, but only the basic version of the Bloch spectral sequence \cite{Bloch}.
\begin{theorem}\label{Pieters} The cohomology of the complex $(L^0_{\rm alt}((S^1)^{\bullet+1})^G, \delta^\bullet)$ is isomorphic to the measurable (equivalently, continuous) group cohomology of $G$. In particular, it vanishes in degrees $\geq 3$.\qed
\end{theorem}

\begin{corollary}[Existence and Uniqueness]\label{ExUn}\item
\begin{enumerate}
\item For every $R \in L^0(\mathcal P_2)$ and $C\in \R$ the system \eqref{PerturbedSpenceAbel} -- \eqref{PerturbedSymmetry} has at most one solution $L^{(R, C)} \in L^0((0,1))$.
\item Given $R \in L^0(\mathcal P_2)$ and $C\in \R$, the system \eqref{PerturbedSpenceAbel} -- \eqref{PerturbedSymmetry} has a solution in $L^0((0,1))$ if and only if $R$ satisfies \eqref{6termlongIntro} and \eqref{RSymmetry} for almost all $(x,y,z) \in \mathcal P_3$.
\item In this case, $L^{(R,C)} = L^{(R-C/2,0)}+C/2$.
\end{enumerate}
\end{corollary}
\begin{proof} Observe first that $L$ satisfies the system \eqref{PerturbedSpenceAbel}-\eqref{PerturbedSymmetry} with right-hand side $R$ and constant $C$ if and only if $\widetilde{L} := L-C/2$ satisfies the same system with right-hand side $\widetilde{R} := R-C/2$ and $C=0$. Indeed we have
\[
\widetilde{L}(1-x) = L(1-x)-C/2 = (C-L(x))-C/2 = C/2-L(x) = -\widetilde{L}(x)
\]
and
\[ \widetilde{L} (x) -\widetilde{L} (y) - \widetilde{L} \left(\frac{x}{y}\right) - \widetilde{L} \left(\frac{y-1}{x-1}\right) +  \widetilde{L} \left(\frac{x(y-1)}{y(x-1)}\right) = R(x,y)-C/2.\]
Also observe that \eqref{6termlongIntro} and \eqref{RSymmetry} hold for $\widetilde{R}$ iff they hold for $R$. Replacing $L$ and $R$ by $\widetilde{L}$ and $\widetilde{R}$ we may thus assume that
$C=0$ and that $L$ satisfies \eqref{PerturbedSymmetry0}.

By \eqref{P1}, the function $L$ satisfies \eqref{PerturbedSymmetry0} if and only if $L \in L^0_{\rm alt}(\mathcal P_1)$. Similarly, \eqref{RSymmetry} amounts to $R \in L^0_{\rm alt}(\mathcal P_2)$. Combining this observation with Corollary \ref{CorDictionary} we see that the corollary amounts to showing that every $4$-cocycle $c$ in the complex $(L^0_{\rm alt}((S^1)^{\bullet+1})^G, \delta^\bullet)$ has a unique primitive. Existence of a primitive is immediate from the vanishing of the $4$th cohomology of this complex, as implied by Theorem \ref{Pieters}. Assume now that $p_1, p_2$ are two primitives of $c$; then the difference $p_1-p_2$ is a $3$-cocycle, hence a $3$-coboundary by applying Theorem \ref{Pieters} again. However, since $G$ acts transitively on $(S^1)^{(3,+)}$, every $2$-cochain and consequently every $3$-coboundary, is trivial.
\end{proof}

\subsection{Continuous bounded cohomology of ${\rm PU}(1,1)$ and boundedness}

We have the following counterparts to Theorem \ref{Pieters} in bounded cohomology.
\begin{theorem}[Burger--Monod, \cite{Burger/On-and-around-the-bounded-cohomology-of-SL2}] The cohomology of the complex $(L^\infty_{\rm alt}((S^1)^{\bullet+1})^G, \delta^\bullet)$ vanishes in degree $3$.\qed
\end{theorem}
\begin{theorem}[\cite{HO1}]\label{Input1} The cohomology of the complex $(L^\infty_{\rm alt}((S^1)^{\bullet+1})^G, \delta^\bullet)$ vanishes in degree $4$.\qed
\end{theorem}
As in the unbounded case, the cohomology of the complex $(L^\infty_{\rm alt}((S^1)^{\bullet+1})^G, \delta^\bullet)$ coincides with the continuous bounded cohomology of $G$ (see \cite{MonodBook}). However, this cohomology is not known beyond degree $4$. Arguing as in the proof of Corollary \ref{ExUn} we obtain:
\begin{corollary}[Boundedness] \label{Bound} Assume that $R \in L^0(\mathcal P_2)$ satisfies \eqref{6termlongIntro} and \eqref{RSymmetry} for almost all $(x,y,z) \in \mathcal P_3$ and let $C \in \R$. Then
\[
L^{(R,C)} \in L^\infty((0,1)) \quad\Leftrightarrow\quad R  \in L^\infty(\mathcal P_2).
\]
\end{corollary}
\begin{proof} Boundedness of $L^{(R,C)}$ implies boundedness of $R$ by \eqref{PerturbedSpenceAbel}. Conversely, if $R \in L^\infty_{\rm alt}(\mathcal P_2)$, then ${\rm ext}(R) \in L^\infty_{\rm alt}((S^1)^5)^G$ has a bounded primitive by Theorem \ref{Input1}.
\end{proof}

\subsection{On measurable lifting and pointwise identities} 
Both the existence and the uniqueness statement in Corollary \ref{ExUn} concern function classes in $L^0$ rather than actual functions. In the bounded case, the existence statement can actually be refined from an almost everywhere identity to a strict pointwise identity. More precisely, let us denote by $\mathcal L^\infty$ the space of bounded measurable functions on a given space. Then the complex $(L^\infty_{\rm alt}((S^1)^{\bullet+1})^G, \delta^\bullet)$ is a quotient of the complex $(\mathcal L^\infty_{\rm alt}((S^1)^{\bullet+1})^G, \delta^\bullet)$, and by \cite{MonodLifting} every cocycle in the latter complex can be lifted to a cocycle in the former complex. In \cite{HO1} we actually established the following refinement of Theorem~\ref{Input1}.
\begin{theorem}[\cite{HO1}]\label{Input2} The cohomology of the complex $(\mathcal L^\infty_{\rm alt}((S^1)^{\bullet+1})^G, \delta^\bullet)$ vanishes in degree $4$.\qed
\end{theorem}
We thus get the following refinement of Corollary \ref{ExUn}.
\begin{corollary} Assume that $R \in \mathcal L^\infty(\mathcal P_2)$ satisfies \eqref{6termlongIntro} and \eqref{RSymmetry} for all (rather than almost all) $(x,y,z) \in \mathcal P_3$ and let $C\in \R$. Then there exists $L \in \mathcal L^\infty_{\rm alt}((0,1))$ which satisfies 
the system \eqref{PerturbedSpenceAbel}-\eqref{PerturbedSymmetry} pointwise (rather than almost everywhere).\qed
\end{corollary}
We do not know whether the uniqueness part of Corollary \ref{ExUn} remains true in the pointwise sense, even for bounded right-hand side.

\section{Explicit formulas, regularity and continuity} \label{SectionExplicit}

\subsection{The function $F^\flat_c$}\label{SecIntegrand} The proof of Theorem \ref{Input2} is constructive, i.e., there is an explicit integral formula for the essentially unique primitive $p \in \mathcal L_{\rm alt}^\infty((S^1)^4)$ of a given cocycle $c \in \mathcal L^\infty_{\rm alt}((S^1)^5)^G$. We now recall this construction in order to establish regularity and continuity of solutions. 

The first step is to construct, out of a given cocycle $c$, a function $F^\flat_c: \Omega \to \R$, where $\Omega$ is the complement of the diagonal in $(0, 2\pi)^2$. This construction proceeds in several steps. We fix a cocycle $c \in \mathcal L_{\rm alt}^\infty((S^1)^5)^G$ and denote by $\Omega$ and $\widehat{\Omega}$ the complements of the respective diagonals in $(0, 2\pi)^2$ and $[0, 2\pi)^2$.
\begin{enumerate}
\item Define $r_c:(0, 2\pi) \to \R$ by
\[
r_c(\varphi) := -\frac 1 2(1-e^{i\varphi}) \cdot \int_{\pi}^\varphi \frac1{1-\cos(\zeta)}  \fint \fint \fint {\sin(\eta-\phi)} c(e^{i\eta}, e^{i\phi}, e^{i\psi}, 1, e^{i\zeta})d\eta d\phi d\psi d\zeta\\
\]
\item Given $\theta_1, \theta_2 \in [0, 2\pi)$ we denote by $\theta_1\ominus \theta_2$ the unique representative of $\theta_1-\theta_2$ modulo $2\pi$ in the interval $[0, 2\pi)$. This defines a measurable function on $[0, 2\pi)^2$, which is real-analytic on $\widehat{\Omega}$.
\item We define $v^\flat_c: \widehat{\Omega} \to \R$ by
\[
v_c^\flat(\theta_1, \theta_2) := {\rm Im}\left(e^{i\theta_1} r_c(\theta_2\ominus\theta_1)\right).
\]
\item Finally we define $F_c^\flat: \Omega \to \R$ by
\[
F_c^\flat(\varphi_1, \varphi_2) :=  \fint \fint \sin(\phi)c(e^{i\eta}, e^{i\phi}, 1, e^{i\varphi_1}, e^{i\varphi_2})d\eta\, d\phi + v_c^\flat(\varphi_1, \varphi_2) - v_c^\flat(0, \varphi_2) + v_c^\flat(0, \varphi_1).
\]
\end{enumerate}
\begin{lemma}\label{Fcflat} 
The above construction has the following properties:
\begin{enumerate}
\item If the cocycle $c \in \mathcal L^\infty_{\rm alt}((S^1)^5)^G$ is of class $C^k$ for some $k \in \mathbb N \cup \{0, \infty, \omega\}$, then so is $F_c^\flat$.
\item If $c_1, c_2 \in \mathcal L^\infty_{\rm alt}((S^1)^5)^G$ are cocycles, then
\[
\|F_{c_1}^\flat-F_{c_2}^\flat\|_{\infty} \leq 4\cdot \|c_1-c_2\|_\infty.
\]
\end{enumerate}
\end{lemma}
\begin{proof} (i) We first show that if $c$ is of class $C^k$, then so is $r_c$. If $k = \omega$ then after exchange of integration and summation this is an immediate consequence of the fundamental theorem of calculus. Thus assume $k \neq \omega$ and observe that both the integrand of $r$ and the bounds of the integral are $C^k$-functions on $(0, 2\pi)$. In particular, the integrand is bounded on every compact interval $[a,b] \subset (0, 2\pi)$. It then follows from the Leibniz integral rule, that $r$ is of class $C^k$ in the interior of each such interval $[a,b]$, hence is of class $C^k$ at every point of $(0, 2\pi)$. Since $\ominus: \widehat \Omega \to \Omega$ is real-analytic, we deduce that also $v^\flat$ is of class $C^k$, and consequently $F_c^\flat$ is of class $C^k$.

(ii) If we denote by $I_{c_j}$ the integrand of $r_{c_j}$, i.e.,
\[
I_{c_j}(\zeta) :=  \frac1{1-\cos(\zeta)}  \fint \fint \fint {\sin(\eta-\phi)} c_j(e^{i\eta}, e^{i\phi}, e^{i\psi}, 1, e^{i\zeta})d\eta d\phi d\psi,
\]
 then
\[
|I_{c_1}(\zeta) - I_{c_2}(\zeta)| \leq \|c_1-c_2\|_\infty \cdot (1-\cos(\zeta))^{-1},
\]
hence
\begin{eqnarray*}
\|v_{c_1}^\flat-v_{c_2}^\flat\|_\infty &\leq& \|r_{c_1}-r_{c_2}\|_\infty\\
&\leq& \sup_{\varphi\in(0, 2\pi)} \|c_1-c_2\|_\infty \cdot \frac{1}{2}\left|(1-e^{i\varphi}) \cdot  \int_{\pi}^\varphi  (1-\cos(\zeta))^{-1} d\zeta\right|\\
&=&\sup_{\varphi\in(0, 2\pi)}  \|c_1-c_2\|_\infty\cdot |\cos(\varphi/2)|\\
&=& \|c_1-c_2\|_\infty.
\end{eqnarray*}
This implies 
\[
\|F_{c_1}^\flat-F_{c_2}^\flat\|_{\infty} \leq 4\cdot \|c_1-c_2\|_\infty.
\]\qedhere
\end{proof}

\subsection{An integration formula for primitives of bounded cocycle}\label{SecIntegrationFormula}
Let $c \in \mathcal L^\infty_{\rm alt}((S^1)^5)^G$ be a cocycle. Following \cite{HO1} we are going to construct a primitive $p$ of $f$ by integration, using the function $F^\flat_c$ defined in the previous subsection. Let us denote by $n_t$ the element of $G$ represented by the matrix
\[
\left(\begin{matrix} 1+\frac{i}{2}t& -\frac{i}{2} t\\  \frac{i}{2} t&1-\frac{i}{2}t \end{matrix}\right),
\]
and let $N = \{n_t\}$ be the corresponding one-parameter subgroup of $G$. Note that the group $N$ acts freely on $\Omega$ and the $N$-orbits intersect the anti-diagonal 
\[\Delta^{\rm op} \deq \bigl\{ (\phi, 2\pi-\phi) \,\big|\, \phi \in (0, 2\pi) \setminus \{ \pi \} \bigr\} \subset \Omega\]
transversally. Thus every $(\varphi_1, \varphi_2) \in \Omega$ can be written uniquely as
\[
(\varphi_1, \varphi_2) = n_{T(\varphi_1, \varphi_2)}.(\Phi(\varphi_1, \varphi_2), 2\pi - \Phi(\varphi_1, \varphi_2))
\]
with $T(\varphi_1, \varphi_2) \in \R$ and $\Phi(\varphi_1, \varphi_2) \in  (0, 2\pi) \setminus \{ \pi \}$. Since
\[
n_t.\varphi = 2 \cdot {\rm arccot}(-t + \cot(\varphi/2)),
\]
we have
\begin{equation}\label{TAndPhi}
T(\varphi_{1},\varphi_{2}) = -\frac{1}{2} \left( \cot\left(\frac{\varphi_{1}}{2}\right) + \cot\left(\frac{\varphi_{2}}{2}\right) \right)\quad \text{and}\quad
\Phi(\varphi_1, \varphi_2) =  2 \cdot {\rm arccot}\left(\frac{1}{2}\left( \cot\left(\frac{\varphi_{1}}{2}\right)- \cot\left(\frac{\varphi_{2}}{2}\right)\right)\right).
\end{equation}
Now the construction of $p$ is as follows.
\begin{enumerate}
\item Define $f^{(c)}_0: \Omega \to \R$ as the integral
\begin{equation}\label{Deff0}
 f_{0}^{(c)}(\varphi_{1},\varphi_{2}) := \int_{0}^{T(\varphi_{1}, \varphi_{2})} F^{\flat}_{c} \bigl( n_{t}.\Phi(\varphi_{1}, \varphi_{2}), n_{t}.(2\pi-\Phi(\varphi_{1}, \varphi_{2})) \bigr) \, dt.
\end{equation}
\item Finally, define $p_c: (S^1)^{(4)} \to \R$ as
\begin{eqnarray}\label{Defpc}
p_c(e^{i\theta_0}, \dots, e^{i\theta_3}) &:=& \fint c(e^{i\theta}, e^{i\theta_0}, \dots, e^{i\theta_3}) d\theta + f^{(c)}_0(\theta_2-\theta_1, \theta_3-\theta_1) - f^{(c)}_0(\theta_2-\theta_0, \theta_3-\theta_0)\nonumber\\&& + f^{(c)}_0(\theta_1-\theta_0, \theta_3-\theta_0) -f^{(c)}_0(\theta_1-\theta_0, \theta_2-\theta_0).\label{Defp}
\end{eqnarray}
\end{enumerate}
It was established in \cite{HO1} that $p_c$ is indeed bounded and a primitive of $c$. Here we observe:
\begin{lemma}\label{LemmaBounds} The above construction has the following properties:
\begin{enumerate}
\item If the cocycle $c \in \mathcal L^\infty_{\rm alt}((S^1)^5)^G$ is of class $C^k$ for some $k \in \mathbb N \cup \{0, \infty, \omega\}$, then so is $p_c$.
\item If $c_1, c_2 \in \mathcal L^\infty_{\rm alt}((S^1)^5)^G$ are cocycles, then
\[
\|p_{c_1}-p_{c_2}\|_{\infty} \leq  \left(1+\frac{16}{\sqrt 3}\right) \cdot \|c_1-c_2\|_\infty.
\]
\end{enumerate}
\end{lemma}
\begin{proof} (i) If $c$ is of class $C^k$, then $F^\flat_c$ is of class $C^k$ by Lemma \ref{Fcflat}. Also, $T$ and $\Phi$ are real analytic by the explicit formulas, in particular of class $C^k$. It then follows from the Leibniz integral rule (respectively the fundamental theorem of calculus if $k = \omega$) that $f_0^{(c)}$, and hence also $p_c$ is of class $C^k$.

(ii)  Observe, firstly, that $3$-transitivity of $G$ on $S^1$ together with $G$-invariance of $p_{c_1}$ and $p_{c_2}$ implies
\[
\|p_{c_1} - p_{c_2}\|_\infty = \sup_{z \in S^1} |p_{c_1}(1, e^{2\pi i/3}, e^{4\pi i/3}, z) - p_{c_2}(1, e^{2\pi i/3}, e^{4\pi i/3}, z)|
\]
Since
\begin{eqnarray*}
p_{c_j}(1, e^{2\pi i/3}, e^{4\pi i/3}, e^{i\xi}) &=&  \fint c(e^{i\theta}, 1, e^{2\pi i/3}, e^{4\pi i/3}, e^{i\xi}) d\theta\\&& + f_{0}^{(c_j)}(2\pi/3,\xi-2\pi/3) - f_{0}^{(c_j)}(4\pi/3,\xi) + f_{0}^{(c_j)}(2\pi/3,\xi) - f_{0}^{(c_j)}(2\pi/3,4\pi/3),
\end{eqnarray*}
this implies
\begin{eqnarray}\label{Continuity1}
\|p_{c_1} -p_{c_2}\|_\infty &\leq& \|c_1-c_2\|_\infty + 3\cdot \sup_{\xi \in (0, 2\pi) \setminus\{2\pi/3\}}|f_0^{(c_1)}(2\pi/3, \xi)-f_0^{(c_2)}(2\pi/3, \xi)| \\
&&+ \sup_{\xi \in (0, 2\pi) \setminus\{4\pi/3\}}|f_0^{(c_1)}(4\pi/3, \xi)-f_0^{(c_2)}(4\pi/3, \xi)|.\nonumber
\end{eqnarray}
Since
\[
f_0^{(c_j)}(\varphi_1, \varphi_2) =  \int_{0}^{T(\varphi_{1}, \varphi_{2})} F^{\flat}_{c_j} \bigl( n_{t}.\Phi(\varphi_{1}, \varphi_{2}), n_{t}.(2\pi-\Phi(\varphi_{1}, \varphi_{2})) \bigr) \, dt
\]
and, by Lemma \ref{Fcflat}, $\|F^{\flat}_{c_1}-F^{\flat}_{c_2}\|_\infty \leq 4 \cdot \|c_1 - c_2\|_\infty $, we have
\[
|f_0^{(c_1)}(\varphi_1, \varphi_2)-f_0^{(c_2)}(\varphi_1, \varphi_2)| \leq 4 \cdot \|c_1 - c_2\|_\infty \cdot |T(\varphi_1, \varphi_2)| = 2 \cdot \|c_1 - c_2\|_\infty \cdot\left| \cot\left(\frac{\varphi_{1}}{2}\right) + \cot\left(\frac{\varphi_{2}}{2}\right) \right| .
\]
Note, however, that this does not yield any useful bound, since $\cot$ is unbounded. To circumvent this problem we recall from \cite{HO1} that the functions $f_0^{(c_j)}$ are invariant under the action of the symmetric group $\mathfrak S_3$ on $\Omega$ in which the generators $s_1$ and $s_2$ act by 
\[
s_{1}.(\phi_{1}, \phi_{2}) = (-\phi_{1}, \phi_{2}-\phi_{1}), \quad s_{2}.(\phi_{1}, \phi_{2}) = (\phi_{2}, \phi_{1}).
\]
Under this action, every point on either of the lines $(2\pi/3, \xi)$ and $(4\pi/3, \xi)$ is equivalent to a point inside the box $[2\pi/3, 5\pi/3] \times [\pi/3, 4\pi/3]$. Inserting everything into \eqref{Continuity1} we thus obtain
\[
\|p_{c_1} -p_{c_2}\|_\infty \leq  \|c_1-c_2\|_\infty \cdot \left(1+8 \cdot \left(\max_{\varphi_1 \in [2\pi/3, 5\pi/3]} \left|\cot\left(\frac{\varphi_1}{2}\right)\right| +\max_{\varphi_2 \in [\pi/3, 4\pi/3]} \left|\cot\left(\frac{\varphi_1}{2}\right|\right) \right|\right)
\]
Since both maxima are equal to $1/\sqrt 3$ we obtain
\[
\|p_{c_1}-p_{c_2}\|_\infty \leq \left(1+\frac{16}{\sqrt 3}\right) \cdot \|c_1-c_2\|_\infty.\qedhere
\]
\end{proof}

\subsection{Regularity and continuity}
We are now in a position to deduce the remaining two parts of Theorem \ref{IntroSummary}.
\begin{corollary}[Regularity and Continuity]\item
 \begin{enumerate}
\item Assume that $R \in L^\infty(\mathcal P_2)$ is essentially bounded and satisfies \eqref{6termlongIntro}, and let $C \in \R$ and $k \in \mathbb N \cup\{0, \infty, \omega\}$. Then $L^{(R, C)}$ is of class $C^k$ if and only if $R$ is of class $C^k$.
\item Assume that $R_1, R_2 \in L^\infty(\mathcal P_2)$ are essentially bounded and satisfy \eqref{6termlongIntro}, and let $C_1, C_2 \in \R$. Then
\[
\|L^{(R_1, C_1)}- L^{(R_2, C_2)}\|_\infty\leq \left(1+\frac{16}{\sqrt 3}\right) \cdot \|R_1-R_2\|_\infty + \left(1+\frac{8}{\sqrt 3}\right) \cdot |C_1-C_2|.
\]
\end{enumerate}
\end{corollary}
\begin{proof} (i) If $L^{(R,C)}$ is of class $C^k$, then $R$ is of class $C^k$ by \eqref{PerturbedSpenceAbel}. Conversely, assume that $R$ is of class $C^k$. Since cross ratio coordinates are real-analytic, it follows that $c := {\rm ext}(R-C/2)$ is of class $C^k$, and thus Lemma \ref{LemmaBounds} shows that its unique primitive $p_c$ is of class $C^k$. Note that by Corollary \ref{CorDictionary} we have $L^{(R-C/2,0)} = {\rm res}(p_c)$. It thus follows from the explicit formula for ${\rm res}$ and real-analyticity of the Cayley transform that $L^{(R-C/2,0)}$ is of class $C^k$. We then deduce with Corollary \ref{ExUn} that $L^{(R,C)} = L^{(R-C/2,0)}+C/2$ is of class $C^k$.

(ii) Assume first that $C = 0$ and let $c_j := {\rm ext}(R_j)$ and $p_j := {\rm ext}(L^{(R_j, 0)})$ We then deduce from Lemma \ref{LemmaBounds} and Corollary \ref{CorDictionary} that
\[\|L^{(R_1, 0)}-L^{(R_2, 0)}\|_\infty \leq \|p_1-p_2\|_\infty \leq \left(1+\frac{16}{\sqrt 3}\right) \cdot \|c_1-c_2\|_\infty = \left(1+\frac{16}{\sqrt 3}\right) \cdot \|R_1-R_2\|_\infty.\]
In the general case we obtain
\begin{eqnarray*}
\|L^{(R_1, C_1)}- L^{(R_2, C_2)}\|_\infty &=& \| L^{(R_1-C_1/2,0)}+C_1/2- L^{(R_2-C_2/2,0)}-C_2/2\|_\infty \\
&\leq&\| L^{(R_1+C_1/2,0)} -  L^{(R_2+C_2/2,0)}\|_\infty  + \frac 1 2 \cdot |C_1-C_2|\\
&\leq&  \left(1+\frac{16}{\sqrt 3}\right) \cdot \|R_1+C_1/2 - R_2-C_2/2\|_\infty + \frac 1 2 \cdot |C_1-C_2|\\
&\leq&  \left(1+\frac{16}{\sqrt 3}\right) \cdot \|R_1-R_2\|_\infty + \left(1+\frac{8}{\sqrt 3}\right) \cdot |C_1-C_2|.
\end{eqnarray*}
\end{proof}

Numerically the constants $1+\frac{16}{\sqrt 3}$ and $1+\frac{8}{\sqrt 3}$ are given by $10.2376...$ and $5.6188...$ respectively.

\subsection{A formula for $L^{(R,C)}$}\label{SecExplicit} We have seen that regularity and continuity of the functions $L^{(R,C)}$ are consequences of the above integration formula. We can actually make this formula more explicit. Although this was not necessary for the above proofs, it may be helpful towards a better understanding of these functions. In the sequel we fix an arbitrary constant $C\in \R$ and a right hand side $R \in L^\infty(\mathcal P_2)$ satisfying \eqref{6termlongIntro}. Define
\[
c:= c^{(R, C)} := {\rm ext}(R-C/2) \quad \text{and} \quad F^{\flat}_{(R,C)} := F^\flat_c.
\]
By Corollary \ref{CorDictionary}, $c \in L^\infty_{\rm alt}((S^1)^5)^G$ is a cocycle and hence $c = dp_c$ for some unique primitive $p_c \in  L^\infty_{\rm alt}((S^1)^4)^G$. Note that an explicit formula for $p_c$ was given in \eqref{Defpc} in terms of the various auxiliary functions defined in Subsections \ref{SecIntegrand} and \ref{SecIntegrationFormula} above.

Once the function $p_c$ has been determined, another application of Corollary \ref{CorDictionary} yields \[L^{(R-C/2, 0)} = {\rm res}(p_c),\] and hence $L^{(R,C)} = {\rm res}(p_c) +C/2$ by Corollary \ref{ExUn}. Explicitly this means that for $x \in (0,1)$ we have
\begin{equation}\label{LRCFormula}
L^{(R,C)}(x) =  p_c(-i, 1, -1, \mathcal C(x)) + C/2.
\end{equation}
Since the formula for $p_c$ involves the arguments of its entries, we need to determine the argument of $\mathcal C(x)$. Since $|\mathcal C(x)|= 1$ we have $\mathcal C(x) = e^{i\theta(x)} $ for some $\theta(x) \in [0, 2\pi)$. In fact, since $x \in (0,1)$ and the Cayley transform maps $(0,1, \infty)$ to $(-1, -i, 1)$, the point $\mathcal C(x)$ lies on the circle segment between $-1$ and $-i$. It follows that $(1, -1, \mathcal C(x), -i)$ is positively oriented and that
\[
0 < \pi < \theta(x) < 3\pi/2.
\]
We have 
\[
\cos(\theta(x)) + i \sin (\theta(x)) = e^{i\theta(x)} = \mathcal C(x) = \frac{x-i}{x+i} = \frac{(x-i)^2}{x^2+1} = \frac{x^2-1}{x^2+1} + i\frac{-2x}{x^2+1},
\]
hence 
\[
\tan(\theta(x)) =\frac{\sin(\theta(x))}{\cos(\theta(x))} =\frac{2x}{1-x^2}.
\]
Since $\theta(x) \in (\pi, 3\pi/2)$ and the main branch of $\arctan$ takes value in $(-\pi/2, \pi/2)$, we have
\begin{equation}\label{thetax}
\theta(x) = \arctan\left(\frac{2x}{1-x^2}\right) + \pi = 2\cdot \arctan(x)+\pi.
\end{equation}
By \eqref{LRCFormula} and \eqref{Defpc}  we now have
\begin{eqnarray*}
L^{(R,C)}(x) = p_c(-i, 1, -1, \mathcal C(x))+ C/2 &=& -p_c(e^{i\cdot 0}, e^{i \cdot \pi},  e^{i\theta(x)}, e^{i\cdot  \frac{3\pi}{2}}) +  C/2\\
&=&-\fint c\left(e^{i\psi}, e^{i\cdot 0}, e^{i\cdot \pi},  e^{i \theta(x)}, e^{i \cdot \frac{3\pi}{2}}\right) d\psi -  f^{(c)}_0\left(\theta(x)-\pi, \pi/2\right)\\
&& + f^{(c)}_0\left(\theta(x), 3\pi/2\right)- f^{(c)}_0\left(\pi, 3\pi/2 \right)+f^{(c)}_0\left(\pi,\theta(x)\right)+C/2,
\end{eqnarray*}
where, by \eqref{Deff0},
\[
 f_{0}^{(c)}(\varphi_{1},\varphi_{2}) := \int_{0}^{T(\varphi_{1}, \varphi_{2})} F^{\flat}_{c} \bigl( n_{t}.\Phi(\varphi_{1}, \varphi_{2}), n_{t}.(2\pi-\Phi(\varphi_{1}, \varphi_{2})) \bigr) \, dt.
\]
Here $T(\varphi_1, \varphi_2)$ and $\Phi(\varphi_1, \varphi_2)$ are given by \eqref{TAndPhi} and $F^\flat_c$ is defined in terms of $c$ in Subsection \ref{SecIntegrand}. It remains to compute the numbers $T(\varphi_1, \varphi_2)$ and $\Phi(\varphi_1, \varphi_2)$ for 
\[
(\varphi_1, \varphi_2) \in \mathcal F := \left\{\left(\theta(x)-\pi, \pi/2\right), \left(\theta(x), 3\pi/2\right), \left(\pi, 3\pi/2\right), \left(\pi, \theta(x)\right)\right\}.
\]
\begin{lemma} The following table provides the values of $T(\varphi_1, \varphi_2)$, $\Phi(\varphi_1, \varphi_2)$, $n_t.\Phi(\varphi_1, \varphi_2)$ and $n_t.(2\pi-\Phi(\varphi_1, \varphi_2))$ for $(\varphi_1, \varphi_2) \in \mathcal F$.

\begin{tabular}{|c|c||c|c||c|c||c|c|}
\hline
$\varphi_1$& $\varphi_2$ & $ \cot\left(\frac{\varphi_{1}}{2}\right)$ & $ \cot\left(\frac{\varphi_{2}}{2}\right)$ & $T(\varphi_1, \varphi_2)$ & $\Phi(\varphi_1, \varphi_2)$ & $n_t.\Phi(\varphi_1, \varphi_2)$ & $n_t.(2\pi-\Phi(\varphi_1, \varphi_2))$\\
\hline
$\theta(x)-\pi$&$\pi/2$&$\frac 1 x$& $1$&$-\frac{1+x}{2x}$& $2\cdot {\rm arccot}\left(\frac{1-x}{2x} \right)$ & $2\cdot {\rm arccot}\left(-t+\frac{1-x}{2x} \right)$&  $2\cdot {\rm arccot}\left(-t-\frac{1-x}{2x}\right)$\\[0.1cm]
$\theta(x)$&$3\pi/2$&$-x$ &$-1$&$\frac{x+1}{2}$& $2\cdot {\rm arccot}\left(\frac{1-x}{2} \right)$ & $2\cdot {\rm arccot}\left(-t+\frac{1-x}{2} \right)$& $2\cdot {\rm arccot}\left(-t-\frac{1-x}{2} \right)$\\[0.1cm]
$\pi$&$3\pi/2$&$0$ & $-1$& $\frac{1}{2}$& $2\cdot {\rm arccot}\left(\frac{1}{2} \right)$ & $2\cdot {\rm arccot}\left(-t+\frac{1}{2} \right)$& $2\cdot {\rm arccot}\left(-t-\frac{1}{2} \right)$\\[0.1cm]
$\pi$&$\theta(x)$&$0$ & $-x$ &$\frac{x}{2}$& $2\cdot {\rm arccot}\left(\frac x 2 \right)$ & $2\cdot {\rm arccot}\left(-t+\frac x 2 \right)$&$2\cdot {\rm arccot}\left(-t-\frac x 2 \right)$ \\[0.1cm]
\hline
\end{tabular}
\end{lemma}

\begin{proof} Concerning the third and fourth column we observe that  $\cot(\pi/4) = 1$, $\cot(\pi/2) = 0$ and $\cot(3\pi/4) = -1$ and deduce from \eqref{thetax} that
\[\cot(\theta(x)/2) = \cot(\arctan(x)+\pi/2) = -\tan(\arctan(x)) = -x,\]
whereas
\[
\cot((\theta(x)-\pi)/2)  = \cot(\arctan(x)) = \tan(\arctan(x))^{-1} = 1/x.
\]
Then the 5th and 6th column are immediate consequences of \eqref{TAndPhi}, and the final two columns then follow from the fomulas
\[
n_t.(2 \cdot {\rm arccot}(\theta)) = 2\cdot  {\rm arccot}(-t+ \theta),
\]
and
\begin{eqnarray*}
n_t.(2\pi - 2 \cdot {\rm arccot}(\theta)) = 2\cdot  {\rm arccot}(-t+\cot(\pi-{\rm arccot}(\theta))) = 2\cdot {\rm arccot(-t-\theta)}.
\end{eqnarray*}
\end{proof}

We can summarize our computations as follows:
\begin{proposition}\label{PropGeneralFormula} For all $x \in (0,1)$ we have
\begin{eqnarray*}
L^{(R,C)}(x) &=& C/2 - \frac{1}{2\pi}\int_0^{2\pi} c^{(R,C)}\left(e^{i\psi}, 1, -1, \mathcal C(x), -i\right) d\psi\\
&& - \int_0^{\frac{-x-1}{2x}} F_{(R,C)}^{\flat}\left(2\cdot {\rm arccot}\left(-t+\frac{1-x}{2x} \right),2\cdot {\rm arccot}\left(-t-\frac{1-x}{2x} \right)\right)dt\\
&& + \int_0^{\frac{x+1}{2}} F_{(R,C)}^{\flat}\left(2\cdot {\rm arccot}\left(-t+\frac{1-x}{2} \right),2\cdot {\rm arccot}\left(-t-\frac{1-x}{2} \right)\right)dt\\
&&-  \int_0^{\frac{1}{2}}  F_{(R,C)}^{\flat}\left(2\cdot {\rm arccot}\left(-t+\frac{1}{2} \right), 2\cdot {\rm arccot}\left(-t-\frac{1}{2} \right)\right)dt\\
&& + \int_0^{\frac x 2} F_{(R,C)}^{\flat}\left(2\cdot {\rm arccot}\left(-t+\frac x 2 \right), 2\cdot {\rm arccot}\left(-t-\frac x 2 \right)\right)dt.
\end{eqnarray*}\qed
\end{proposition}
We observe for later use that
\[
\mathcal F \subset \Omega^+ := \left\{(\varphi_1, \varphi_2) \in (0, 2\pi)^2\mid \varphi_1 < \varphi_2 \right\},
\]
hence in order to compute $L^{(R,C)}(x)$ we only need to know $F^{\flat}_{(R,C)}$ on $\Omega^+$. Similarly, we only need to compute the integral
\[
 \frac{1}{2\pi}\int_0^{2\pi} c^{(R,C)}\left(e^{i\psi},e^{i\phi_0}, e^{i\phi_1}, e^{i\phi_2}, e^{i\phi_3} \right) d\psi
\]
for $0\leq \phi_0< \phi_1 < \phi_2< \phi_3 < 2\pi$. This will allow us to avoid some technical case distinctions later on.

\section{Rogers' dilogarithm revisited}\label{SecRogers}

\subsection{The cocycle associated with $L_2$ and the general strategy}
By definition, Rogers' dilogarithm $L_2$ satisfies the Spence-Abel equation with right-hand side $R \equiv 0$ and the reflection symmetry with constant $C = \zeta(2)$. In our previous notation this means that $L_2 = L^{(0, \zeta(2))}$. Thus the relevant cocycle entering into the formula for $L_2$ in Proposition \ref{PropGeneralFormula} is $c := c^{(0, \zeta(2))}$. By definition, the cocycle $c$ is constant on $(S^1)^{(5,+)}$ where it takes the value $-\zeta(2)/2$. Given $\underline{z} = (z_0, \dots, z_4) \in (S^1)^{(5)}$ let $\sigma$ be a permutation of $\{0, \dots, 5\}$ such that  $(z_{\sigma(0)}, \dots, z_{\sigma(4)}) \in (S^1)^{(5,+)}$. Then $\sigma$ is determined up to a cyclic (hence even) permutation, and thus the sign $(-1)^\sigma$ depends only on $\underline{z}$. If we denote this sign by $(-1)^{\underline z}$, then
\[
c(\underline{z}) = -\frac{\zeta(2)}{2} \cdot (-1)^{\underline{z}}.
\]
In particular, $c$ is a multiple of ${\rm or} \cup {\rm or}$, where ${\rm or}$ is the orientation cocycle. Our goal for the rest of this section is to make the formula in Proposition \ref{PropGeneralFormula} more explicit for this specific cocycle. Our computation proceeds in three stages: Firstly, we compute in Subsection \ref{SecBasicInt} some basic integrals of $c$. Using these basic integrals we derive in Subsection \ref{SecIntL2} the integrand appearing in Proposition \ref{PropGeneralFormula} for our specific $c$. Finally, in Subsection \ref{L2Final} we put everything together and derive the formula for $L_2$ stated in the introduction.

\subsection{Basic integrals}\label{SecBasicInt}
We now turn to the first step of the outline given in the previous subsection and compute some basic integrals related to the cocycle $c(\underline{z}) = -\frac{\zeta(2)}{2} \cdot (-1)^{\underline{z}}$.
\begin{lemma}\label{I1} If $0\leq \theta_0 < \dots < \theta_3 < 2\pi$, then
\[
I_1(e^{i\theta_0}, \dots, e^{i\theta_3}) := \fint c(e^{i\psi}, e^{i\theta_0}, \dots, e^{i\theta_3}) d\psi = -\frac{\zeta(2)}{2\pi}\left(\theta_0-\theta_1+\theta_2 - \theta_3 + \pi \right).
\]
\end{lemma}
\begin{proof} Let $\underline{z} := (e^{i\psi}, e^{i\theta_0}, \dots, e^{i\theta_3})$. If $\psi \in (0, \theta_0) \cup (\theta_1, \theta_2) \cup (\theta_3, 2\pi)$, then $(-1)^{\underline{z}} = 1$, and if $\psi \in (\theta_0, \theta_1) \cup (\theta_2, \theta_3)$ then $(-1)^{\underline{z}} = -1$. We thus obtain
\begin{eqnarray*}
I_1(e^{i\theta_0}, \dots, e^{i\theta_3}) &=& -\frac{\zeta(2)}{4\pi}\left(\int_0^{\theta_0} d\psi-\int_{\theta_0}^{\theta_1} d\psi+\int_{\theta_1}^{\theta_2}d\psi-\int_{\theta_2}^{\theta_3}d\psi+\int_{\theta_3}^{2\pi}d\psi\right)\\
&=&-\frac{\zeta(2)}{4\pi}\left(\theta_0 - (\theta_1-\theta_0)+(\theta_2-\theta_1)- (\theta_3-\theta_2) +  (2\pi-\theta_3)\right)\\
&=&-\frac{\zeta(2)}{2\pi}\left(\theta_0-\theta_1+\theta_2 - \theta_3 + \pi \right).
\end{eqnarray*}
\end{proof}
\begin{lemma}\label{I2} If $0<\theta_1 < \theta_2< 2\pi$, then
\[
I_2(e^{i\theta_1}, e^{i\theta_2}) := \fint \fint e^{i\eta}c( e^{i\psi}, e^{i\eta}, 1, e^{i\theta_1}, e^{i\theta_2})d\eta d\psi = -i\cdot \frac{\zeta(2)}{2\pi^2}\cdot ((e^{i\theta_2}-1)(\theta_1-\pi)-(e^{i\theta_1}-1)(\theta_2-\pi)-\pi).
\]
\end{lemma}
\begin{proof} By Lemma \ref{I1} we have
\[
-\frac{2\pi}{\zeta(2)}\cdot I_1(e^{i\eta}, 1, e^{i\theta_1}, e^{i\theta_2}) = \left\{\begin{array}{ll} -\left(0-\eta+\theta_1 - \theta_2 + \pi \right),&\eta \in (0, \theta_1)\\
+\left(0-\theta_1+\eta - \theta_2 + \pi \right),&\eta \in (\theta_1, \theta_2)\\
-\left(0-\theta_1+\theta_2- \eta + \pi \right),&\eta \in (\theta_2, 2\pi),\\
\end{array}\right.
\]
hence
\begin{eqnarray*}
-\frac{4\pi^2}{\zeta(2)}\cdot I_2(e^{i\theta_1}, e^{i\theta_2}) &=& 2\pi \cdot \fint e^{i\eta} \cdot \frac{2\pi}{\zeta(2)}\cdot I_1(e^{i\eta}, 1, e^{i\theta_1}, e^{i\theta_2})d\eta\\
&=& -\int_0^{\theta_1} e^{i\eta} (-\eta+\theta_1 - \theta_2 + \pi)d\eta +\int_{\theta_1}^{\theta_2} e^{i\eta}(-\theta_1+\eta - \theta_2 + \pi ) d\eta\\
&&-\int_{\theta_2}^{2\pi}e^{i\eta}(-\theta_1+\theta_2- \eta + \pi)d\eta\\
&=& \int_0^{2\pi}\eta e^{i\eta} d\eta  -(\theta_1 - \theta_2 + \pi)\int_0^{\theta_1} e^{i\eta} d\eta+(-\theta_1 - \theta_2 + \pi ) \int_{\theta_1}^{\theta_2} e^{i\eta}d\eta\\
&&-(-\theta_1+\theta_2+ \pi)\int_{\theta_2}^{2\pi} e^{i\eta}d\eta\\
&=&-2\pi i + i(\theta_1 - \theta_2 + \pi)(e^{i\theta_1}-1) -i(-\theta_1 - \theta_2 + \pi )(e^{i\theta_2}-e^{i\theta_1})\\
&& + i(-\theta_1+\theta_2+ \pi)(1-e^{i\theta_2})\\
&=& 2i(-\pi+\theta_1e^{i\theta_2}-\theta_2e^{i\theta_1}-\theta_1+\pi e^{i\theta_1} + \theta_2 - \pi e^{i\theta_2})\\
&=& 2i((e^{i\theta_2}-1)(\theta_1-\pi)-(e^{i\theta_1}-1)(\theta_2-\pi)-\pi).\qedhere
\end{eqnarray*}
\end{proof}

\begin{lemma} For all $\theta \in (0, 2\pi)$ we have
\[
I_3(e^{i\theta}) := \fint \fint \fint e^{i(\eta-\phi)}c( e^{i\eta}, e^{i\phi}, e^{i\psi}, 1, e^{i\theta})d\eta d\phi d\psi = -i\cdot \frac{\zeta(2)}{2\pi^2} \cdot(2\sin(\theta)+\theta -\pi).
\]
\end{lemma}
\begin{proof} We have
\[
I_3(e^{i\theta}) = \fint e^{-i\phi} \left( \fint \fint e^{i\eta} c(e^{i\psi},  e^{i\eta}, 1,  e^{i\theta}, e^{i\phi})d\eta  d\psi \right)d\phi = \fint e^{-i\phi} I_2(e^{i\theta}, e^{i\phi})d\phi,
\]
where by Lemma \ref{I2},
\[
-\frac{2\pi^2}{\zeta(2) \cdot i} \cdot I_2(e^{i\theta}, e^{i\phi}) = \left\{\begin{array}{ll} -(e^{i\theta}-1)(\phi-\pi)+(e^{i\phi}-1)(\theta-\pi)+\pi, & \phi \in (0, \theta)\\
 (e^{i\phi}-1)(\theta-\pi)-(e^{i\theta}-1)(\phi-\pi)-\pi, & \phi \in (\theta, 2\pi).
 \end{array}\right.
\]
Using that
\[
\int_0^\theta \phi e^{-i\phi}d\phi = -1+e^{-i\theta}(1+i\theta),\quad \int_\theta^{2\pi} \phi e^{-i\phi}d\phi = 2\pi i + 1 - e^{-i\theta}(1+i\theta),
\]
\[
\int_0^\theta e^{-i\phi}d\phi =ie^{-i\theta}-i \quad \text{and} \quad \int_\theta^{2\pi} e^{-i\phi}d\phi = i-ie^{-i\theta},
\]
we conclude that
\begin{eqnarray*} 
-\frac{4\pi^3}{\zeta(2) \cdot i} \cdot I_3(e^{i\theta}) &=& 2\pi \cdot \fint e^{-i\phi} \cdot \frac{2\pi^2}{\zeta(2) \cdot i} \cdot I_2(e^{i\theta}, e^{i\phi}) d\phi = \int_0^{2\pi}  e^{-i\phi} \cdot \frac{2\pi^2}{\zeta(2) \cdot i} \cdot I_2(e^{i\theta}, e^{i\phi}) d\phi \\
&=& \int_0^{\theta} -e^{-i\phi}(e^{i\theta}-1)(\phi-\pi)+e^{-i\phi}(e^{i\phi}-1)(\theta-\pi)+e^{-i\phi}\pi d\phi \\
&&+ \int_{\theta}^{2\pi} e^{-i\phi}(e^{i\phi}-1)(\theta-\pi)-e^{-i\phi}(e^{i\theta}-1)(\phi-\pi)-e^{-i\phi}\pi d\phi\\
&=& -(e^{i\theta}-1) \cdot  \int_0^{\theta} \phi e^{-i\phi} d\phi +(e^{i\theta}\pi - \theta +\pi) \cdot  \int_0^{\theta} e^{-i\phi} d\phi +(\theta - \pi) \cdot \int_0^\theta d\phi\\
&&- (e^{i\theta}-1) \cdot  \int_{\theta}^{2\pi} \phi e^{-i\phi} d\phi + (e^{i\theta}\pi-\theta - \pi)  \cdot  \int_{\theta}^{2\pi} e^{-i\phi} d\phi +(\theta-\pi) \cdot \int_\theta^{2\pi} d\phi\\
&=& 2\pi(\theta - \pi) - (e^{i\theta}-1) \cdot 2\pi i + [(e^{i\theta}\pi - \theta +\pi)-(e^{i\theta}\pi-\theta - \pi)](ie^{-i\theta}-i )\\
&=& 2\pi \cdot (\theta - \pi-(ie^{i\theta}-i)+(ie^{-i\theta}-i )) = 2\pi \cdot (2\sin (\theta) + \theta - \pi).\\
\end{eqnarray*}\qedhere
\end{proof}

\subsection{The integrand}\label{SecIntL2} We now turn to the second step of our outline, the computation of the integrand $F^\flat_c$ appearing in Proposition \ref{PropGeneralFormula} for our specific cocycle $c(\underline{z}) =  -\frac{\zeta(2)}{2} \cdot (-1)^{\underline{z}}$. We recall that $F_c^\flat$ was defined in Subsection \ref{SecIntegrand}, and in the sequel we will use the notation of that subsection. In particular, the functions $r_c$ and $v_c^\flat$ are defined as there.
\begin{lemma} If $0\leq \theta_1 < \theta_2 < 2\pi$, then $v_c^\flat(\theta_1, \theta_2) = v_1(\theta_1, \theta_2) + v_2(\theta_1, \theta_2)$, where
\begin{eqnarray*}
v_1(\theta_1, \theta_2) &:=& \frac{\zeta(2)}{4\pi^2} \cdot (\cos(\theta_1)+\cos(\theta_2)) \cdot ({\theta_2-\theta_1}-\pi),\\
v_2(\theta_1, \theta_2) &:=&  \frac{3\cdot \zeta(2)}{2\pi^2} \cdot  (\sin(\theta_1)-\sin(\theta_2))  \cdot \log(\sin((\theta_2-\theta_1)/2)).
\end{eqnarray*}
\end{lemma}
\begin{proof} Using the fact that $- i \cdot \cos(\eta-\varphi)$ is symmetric in $\eta$ and $\varphi$ and hence integrates to $0$ against $c(e^{i\eta}, e^{i\phi}, e^{i\psi}, 1, e^{i\varphi})$ we can write the function $r_c$ as 
\begin{eqnarray*}
r_c(\varphi) &=&-\frac 1 2(1-e^{i\varphi}) \cdot \int_{\pi}^\varphi \frac1{1-\cos(\zeta)}  \fint \fint \fint {\sin(\eta-\varphi)} c(e^{i\eta}, e^{i\phi}, e^{i\psi}, 1, e^{i\zeta})d\eta d\phi d\psi d\zeta\\
&=& \frac 1 2(1-e^{i\varphi}) \cdot \int_{\pi}^\varphi \frac1{1-\cos(\zeta)}  \fint \fint \fint ie^{i(\eta-\varphi)} c(e^{i\eta}, e^{i\phi}, e^{i\psi}, 1, e^{i\zeta})d\eta d\phi d\psi d\zeta\\
&=& \frac i 2(1-e^{i\varphi}) \cdot \int_{\pi}^\varphi \frac1{1-\cos(\zeta)}  I_3(e^{i\zeta}) d\zeta\\
&=& \frac{\zeta(2)}{4\pi^2} \cdot (1-e^{i\varphi}) \cdot \int_{\pi}^\varphi \frac{2\sin(\zeta)+\zeta -\pi}{1-\cos(\zeta)} d\zeta\\
&=&  \frac{\zeta(2)}{4\pi^2} \cdot (1-e^{i\varphi}) \cdot \left.\left((\pi-\zeta)\cdot \cot\left(\frac \zeta 2\right)+6 \log\left(\sin\left(\frac \zeta 2\right)\right) \right)\right|_{\zeta = \pi}^\varphi\\
&=&  \frac{\zeta(2)}{4\pi^2} \cdot (1-e^{i\varphi}) \cdot ((\pi-\varphi)\cdot \cot(\varphi/2)+6 \log(\sin(\varphi/2))).
\end{eqnarray*}
We deduce that
\begin{eqnarray*}
v_c^\flat(\theta_1, \theta_2) &=& {\rm Im}\left( -\frac{\zeta(2)}{4\pi^2} \cdot e^{i\theta_1} \cdot (1-e^{i{(\theta_2\ominus\theta_1)}}) \cdot ({\theta_2\ominus\theta_1}-\pi) \cdot \cot((\theta_2\ominus\theta_1)/2)\right.\\
&&\left. +\frac{3\cdot \zeta(2)}{2\pi^2} \cdot e^{i\theta_1} \cdot  (1-e^{i{\theta_2\ominus\theta_1}})  \cdot \log(\sin((\theta_2\ominus\theta_1)/2)) \right)\\
&=& -{\rm Im}\left( \frac{\zeta(2)}{4\pi^2} \cdot (e^{i\theta_1}-e^{i{\theta_2}}) \cdot ({\theta_2\ominus\theta_1}-\pi) \cdot \cot((\theta_2\ominus \theta_1)/2)\right.\\
&&\left. +\frac{3\cdot \zeta(2)}{2\pi^2} \cdot  (e^{i\theta_1}-e^{i{\theta_2}}) \cdot \log(\sin((\theta_2\ominus\theta_1)/2)) \right)\\
&=& -\frac{\zeta(2)}{4\pi^2} \cdot (\sin(\theta_1)-\sin(\theta_2)) \cdot ({\theta_2\ominus\theta_1}-\pi) \cdot \cot((\theta_2\ominus \theta_1)/2)\\
&&  +\frac{3\cdot \zeta(2)}{2\pi^2} \cdot  (\sin(\theta_1)-\sin(\theta_2))  \cdot \log(\sin((\theta_2\ominus\theta_1)/2)).
\end{eqnarray*}
If $0\leq \theta_1 < \theta_2 < 2\pi$ we may replace $\ominus$ by $-$ in this formula. Then the second summand is just $v_2(\theta_1, \theta_2)$. To see that the first summand equals $v_1(\theta_1, \theta_2)$ we use the trigonometric identity
\begin{eqnarray*}
-(\sin(\theta_1)-\sin(\theta_2)) \cdot \cot((\theta_2- \theta_1)/2) &=& -2 \cdot \sin((\theta_1-\theta_2)/2) \cdot \cos((\theta_1+\theta_2)/2) \cdot \cot((\theta_2- \theta_1)/2)\\
&=& 2 \cdot  \cos((\theta_2+\theta_1)/2) \cdot  \cos((\theta_2- \theta_1)/2)\\
&=& 2 \cdot (\cos^2(\theta_2/2) - \sin^2(\theta_1/2))\\
&=& 2 \cdot (1/2 + 1/2 \cos(\theta_2) - (1/2 -1/2 \cos(\theta_1)))\\
&=&  \cos(\theta_1)+\cos(\theta_2).
\end{eqnarray*}
This finishes the proof.
\end{proof}
\begin{corollary}\label{Fcflat2} If $c(\underline{z}) = -\frac{\zeta(2)}{2} \cdot (-1)^{\underline{z}}$ and $0<\varphi_1 <\varphi_2< 2\pi$, then
\begin{eqnarray*}
F_c^\flat(\varphi_1, \varphi_2)&=&  -\frac{\zeta(2)}{4\pi^2} \cdot \left( (3\varphi_1-2\pi)(\cos(\varphi_2)-1) - (3\varphi_2-4\pi)(\cos(\varphi_1)-1)\right) \\
&& +  \frac{3\zeta(2)}{2\pi^2} \cdot \left(\sin(\varphi_1) \cdot \log\left(\frac{\sin((\varphi_2-\varphi_1)/2)}{\sin(\varphi_1/2)}\right)- \sin(\varphi_2) \cdot \log\left(\frac{\sin((\varphi_2-\varphi_1)/2)}{\sin(\varphi_2/2)}\right)\right).
\end{eqnarray*}
\end{corollary}
\begin{proof} We recall that
\[
F_c^\flat(\varphi_1, \varphi_2) =  \fint \fint \sin(\phi)c(e^{i\eta}, e^{i\phi}, 1, e^{i\varphi_1}, e^{i\varphi_2})d\eta\, d\phi + v_c^\flat(\varphi_1, \varphi_2) - v_c^\flat(0, \varphi_2) + v_c^\flat(0, \varphi_1).
\]
 If $0\leq \varphi_1 <\varphi_2< 2\pi$, then the first summand is given by
\begin{eqnarray*}
  \fint \fint \sin(\phi)c(e^{i\eta}, e^{i\phi}, 1, e^{i\varphi_1}, e^{i\varphi_2})d\eta\, d\phi &=&  {\rm Re}(-i \cdot I_2(e^{i\varphi_1}, e^{i\varphi_2}))\\ &=& -\frac{\zeta(2)}{2\pi^2}\cdot ((\cos(\varphi_2)-1)(\varphi_1-\pi)-(\cos(\varphi_1)-1)(\varphi_2-\pi)-\pi).
\end{eqnarray*}
We regroup the remaining terms as
\[
v_c^\flat(\varphi_1, \varphi_2) - v_c^\flat(0, \varphi_2) + v_c^\flat(0, \varphi_1) = (v_1(\varphi_1, \varphi_2) - v_1(0, \varphi_2) + v_1(0, \varphi_1)) + (v_2(\varphi_1, \varphi_2) - v_2(0, \varphi_2) + v_2(0, \varphi_1)).
\]
The $v_1$-terms then sum up to
\begin{eqnarray*}
&& \frac{4\pi^2}{\zeta(2)} \cdot  (v_1(\varphi_1, \varphi_2) - v_1(0, \varphi_2) + v_1(0, \varphi_1))\\
&=& (\cos(\varphi_1)+\cos(\varphi_2))\cdot(\varphi_2-\varphi_1-\pi)-(1+\cos(\varphi_2))\cdot(\varphi_2-\pi)+(1+\cos(\varphi_1))\cdot(\varphi_1-\pi)\\
&=& \cos(\varphi_1)\varphi_2 - \cos(\varphi_1)\varphi_1 - \cos(\varphi_1)\pi+\cos(\varphi_2)\varphi_2 - \cos(\varphi_2)\varphi_1 - \cos(\varphi_2)\pi\\
&& -\varphi_2+\pi-\cos(\varphi_2)\varphi_2 +\cos(\varphi_2)\pi+\varphi_1-\pi+\cos(\varphi_1)\varphi_1 -\cos(\varphi_1)\pi\\
&=& \cos(\varphi_1)\varphi_2 - 2\cos(\varphi_1)\pi -\cos(\varphi_2)\varphi_1-\varphi_2+\varphi_1,
\end{eqnarray*}
whereas the $v_2$-terms sum up to
\begin{eqnarray*}
&& \frac{2\pi^2}{3\cdot \zeta(2)} \cdot  (v_2(\varphi_1, \varphi_2) - v_2(0, \varphi_2) + v_2(0, \varphi_1))\\ 
&=&  (\sin(\varphi_1)-\sin(\varphi_2))  \cdot \log(\sin((\varphi_2-\varphi_1)/2))+\sin(\varphi_2)  \cdot \log(\sin(\varphi_2/2))-\sin(\varphi_1)  \cdot \log(\sin(\varphi_1/2))\\
&=&   \sin(\varphi_1) \cdot ( \log(\sin((\varphi_2-\varphi_1)/2))- \log(\sin(\varphi_1/2)))-\sin(\varphi_2)\cdot (\log(\sin((\varphi_2-\varphi_1)/2))- \log(\sin(\varphi_2/2)))\\
&=& \sin(\varphi_1) \cdot \log\left(\frac{\sin((\varphi_2-\varphi_1)/2)}{\sin(\varphi_1/2)}\right)- \sin(\varphi_2) \cdot \log\left(\frac{\sin((\varphi_2-\varphi_1)/2)}{\sin(\varphi_2/2)}\right)
\end{eqnarray*}
We thus obtain
\begin{eqnarray*}
 \frac{4\pi^2}{\zeta(2)} \cdot F_c^\flat(\varphi_1, \varphi_2) &=& -2\cdot(\cos(\varphi_2)-1)(\varphi_1-\pi)+2\cdot(\cos(\varphi_1)-1)(\varphi_2-\pi)+2\pi\\
 &&  +\cos(\varphi_1)\varphi_2 - 2\cos(\varphi_1)\pi -\cos(\varphi_2)\varphi_1-\varphi_2+\varphi_1\\
 &&+ 6 \sin(\varphi_1) \cdot \log\left(\frac{\sin((\varphi_2-\varphi_1)/2)}{\sin(\varphi_1/2)}\right)-6\sin(\varphi_2) \cdot \log\left(\frac{\sin((\varphi_2-\varphi_1)/2)}{\sin(\varphi_2/2)}\right)\\
 &=& -3\varphi_1(\cos(\varphi_2)-1)+3\varphi_2(\cos(\varphi_1)-1)-2\pi(2\cos(\varphi_1)+\cos(\varphi_2) - 1)\\
  &&+ 6 \sin(\varphi_1) \cdot \log\left(\frac{\sin((\varphi_2-\varphi_1)/2)}{\sin(\varphi_1/2)}\right)- 6\sin(\varphi_2) \cdot \log\left(\frac{\sin((\varphi_2-\varphi_1)/2)}{\sin(\varphi_2/2)}\right).
\end{eqnarray*}
Finally note that
\begin{eqnarray*}
&&-3\varphi_1(\cos(\varphi_2)-1)+3\varphi_2(\cos(\varphi_1)-1)-2\pi(2\cos(\varphi_1)+\cos(\varphi_2) - 1)\\
&=& -3\varphi_1(\cos(\varphi_2)-1)+3\varphi_2(\cos(\varphi_1)-1)-4\pi(\cos(\varphi_1)-1) + 2\pi(\cos(\varphi_2)-1)\\
&=& -(3\varphi_1-2\pi)(\cos(\varphi_2)-1) + (3\varphi_2-4\pi)(\cos(\varphi_1)-1).
\end{eqnarray*}
\end{proof}
\subsection{The final formula}\label{L2Final} Combining Proposition \ref{PropGeneralFormula}, Lemma \ref{I1} and Corollary \ref{Fcflat2} we finally obtain:
\begin{theorem}\label{ThmFinalFormula} The Rogers dilogarithm $L_2$ is given by the formula
\begin{eqnarray*}
L_2(x) &=& \frac{\zeta(2)}{2} + \frac{\zeta(2)}{2\pi}\left( \arctan\left(\frac{2x}{1-x^2}\right) - \frac{\pi}{2} \right)\\
&& - \int_0^{\frac{-x-1}{2x}} F_{c}^{\flat}\left(2\cdot {\rm arccot}\left(-t+\frac{1-x}{2x} \right),2\cdot {\rm arccot}\left(-t-\frac{1-x}{2x} \right)\right)dt\\
&& + \int_0^{\frac{x+1}{2}} F_{c}^{\flat}\left(2\cdot {\rm arccot}\left(-t+\frac{1-x}{2} \right),2\cdot {\rm arccot}\left(-t-\frac{1-x}{2} \right)\right)dt\\
&&-  \int_0^{\frac{1}{2}}  F_{c}^{\flat}\left(2\cdot {\rm arccot}\left(-t+\frac{1}{2} \right), 2\cdot {\rm arccot}\left(-t-\frac{1}{2} \right)\right)dt\\
&& + \int_0^{\frac x 2} F_{c}^{\flat}\left(2\cdot {\rm arccot}\left(-t+\frac x 2 \right), 2\cdot {\rm arccot}\left(-t-\frac x 2 \right)\right)dt,
\end{eqnarray*}
where 
\begin{eqnarray*}
F_c^\flat(\varphi_1, \varphi_2)&=& -\frac{\zeta(2)}{4\pi^2} \cdot \left( (3\varphi_1-2\pi)(\cos(\varphi_2)-1) - (3\varphi_2-4\pi)(\cos(\varphi_1)-1)\right) \\
&& +  \frac{3\zeta(2)}{2\pi^2} \cdot \left(\sin(\varphi_1) \cdot \log\left(\frac{\sin((\varphi_2-\varphi_1)/2)}{\sin(\varphi_1/2)}\right)- \sin(\varphi_2) \cdot \log\left(\frac{\sin((\varphi_2-\varphi_1)/2)}{\sin(\varphi_2/2)}\right)\right).
\end{eqnarray*}\qed
\end{theorem}

\bibliographystyle{amsplain}

\begin{thebibliography}{10}

\bibitem{Abel}
N.~H.~Abel, \emph{Note sur la fonction $\psi(x) = x+\frac{x^2}{2^2}+\frac{x^3}{3^2} + \dots +\frac{x^n}{n^2}  + \dots$}, unpublished manuscript, 1826, published poshumously in:  L.~Sylow and S.~Lie, Oeuvres complètes de Niels Henrik Abel, Grøndahl \& Søn, Christiania [Oslo], 1881, 189--193.

\bibitem{Bloch}
S.~J.~Bloch, \emph{Higher regulators, algebraic {$K$}-theory, and zeta
  functions of elliptic curves}, CRM Monograph Series, vol.~11, American
  Mathematical Society, Providence, RI, 2000.

\bibitem{K1}
\bysame, \emph{Applications of the dilogarithm function in algebraic K-theory and algebraic geometry}, in: Proceedings of the International Symposium on Algebraic Geometry (Kyoto Univ., Kyoto, 1977), Kinokuniya Book Store, Tokyo, 1978, 103--114.

\bibitem{HyperbolicVolume}
A.~Borel, \emph{Commensurability classes and volumes of hyperbolic 3-manifolds},
Ann. Scuola Norm. Sup. Pisa Cl. Sci. (4) 8 (1981), no. 1, 1--33. 

\bibitem{BuMo}
M.~Burger and N.~Monod, \emph{Bounded cohomology of lattices in higher rank
  {L}ie groups}, J. Eur. Math. Soc. (JEMS) \textbf{1} (1999), no.~2, 199--235.

\bibitem{Burger/On-and-around-the-bounded-cohomology-of-SL2}
\bysame, \emph{On and around the bounded cohomology of {$SL_2$}}, in: Rigidity in {D}ynamics and {G}eometry, Proceedings of the {P}rogramme on {E}rgodic {T}heory, {G}eometric {R}igidity and {N}umber {T}heory held in {C}ambridge, {J}anuary 5--{J}uly 7, 2000, edited by M.~Burger and A.~Iozzi, Springer, 2002, 19--37.

\bibitem{BOT}
M.~Burger, N.~Ozawa, A.~Thom, \emph{On Ulam stability}, Israel Journal of Mathematics, 2013 (1), Vol. 193, 109--129.

\bibitem{Dupont}
J.~L.~Dupont, \emph{The dilogarithm as a characteristic class for flat bundles}, in: 
Proceedings of the Northwestern conference on cohomology of groups (Evanston, Ill., 1985). 
J. Pure Appl. Algebra 44 (1987), no. 1-3, 137--164.

\bibitem{CFT1}
\bysame, \emph{Dilogarithm identities in conformal field theory and group homology}, in: 
Geometry and global analysis (Sendai, 1993), Tohoku Univ., Sendai, 1993, 191--195. 

\bibitem{CFT2}
J.~L.~Dupont and C.-H.~Sah, \emph{Dilogarithm identities in conformal field theory and group homology},
Comm. Math. Phys. 161 (1994), no. 2, 265--282.

\bibitem{Euler}
L.~Euler, \emph{Institutionum calculi differentialis}, Impensis Academiae Imperialis Scientiarum, Petropolis, 1768.

\bibitem{FK}
K.~Fujiwara and M.~Kapovich, \emph{On quasihomomorphisms with noncommutative targets}, Preprint, 2013.

\bibitem{Gon}
A.~B.~Goncharov, \emph{The classical polylogarithms, algebraic {K}-theory and $\zeta_F(n)$} 
The Gel'fand Mathematical Seminars, 1990--1992, Birkh\"auser Boston, 1993, 113--135.

\bibitem{Gromov}
M.~Gromov, \emph{Volume and bounded cohomology}, Inst. Hautes \'Etudes Sci.
  Publ. Math. (1982), no.~56, 5--99 (1983).

\bibitem{HO1}
T.~Hartnick and A.~Ott, \emph{Bounded cohomology via partial differential equations, I}, Geom. Top. 19-6 (2015), 3603--3643.

\bibitem{HS}
T.~Hartnick and P.~Schweitzer, \emph{On quasioutomorphism groups of free groups and their transitivity properties}, J. Algebra 450 (2016), 242--281.

\bibitem{Spence}
J.~F.~W.~Herschel, \emph{Mathematical Essays by the late William Spence, Esq.}, Oliver and Boyd, Edinburgh,1819.

\bibitem{Hill}
C.~J.~Hill, \emph{\"Uber die Integration logarithmisch-rationaler Differentiale}, J. Reine Angew.
Math. 3 (1828), 101--159.

\bibitem{Hyers}
D.~H.~Hyers, \emph{On the stability of the linear functional equation}, Proc. Natl. Acad. Sci. USA, 27(1941), 222--224.


\bibitem{HIR}
D.~H.~Hyers, G. Isac and T. Rassias, \emph{Stability of functional equations in several variables}, Progress in Non-linear Differential Equations and their
              Applications, 34, Birkh\"auser, Boston, 1998.

\bibitem{Ivanov}
N.~V.~Ivanov, \emph{Foundations of the theory of bounded cohomology}, Journal of Soviet Mathematics, 1987, vol. 37, no. 3, 1090--1115.

\bibitem{Jung}
S.-M.~Jung, \emph{Hyers-{U}lam-{R}assias stability of functional equations in
              nonlinear analysis}, Springer Optimization and Its Applications 48, Springer, New York, 2011.		

\bibitem{CFT4}
A.~N.~Kirillov, \emph{Dilogarithm identities, partitions, and spectra in conformal field theory},
St. Petersburg Math. J. 6 (1995), no. 2, 327--348. 

\bibitem{Kannappan}
P.~Kannappan, \emph{Functional equations and inequalities with applications}, Springer Monographs in Mathematics, Springer, New York, 2009.

\bibitem{Kummer}
E.~E.~Kummer, \emph{\"Uber die Transcendenten, welche aus wiederholten Integrationen rationaler Formeln entstehen.}
J. Reine Angew. Math. 21 (1840), 74--90. 

\bibitem{Leibniz}
G.~W.~Leibniz, \emph{Letter to J. Bernoulli from November 9th, 1696}, in: S\"amtliche Brief und Schriten, Reihe III, Band 7 (ed. J.G O'Hara, C. Wahl, R. Kr\"omer, H. Sefrin-Weis), Akademie-Verlag, Berlin 2011, 178--179.
\bibitem{K2}
S.~Lichtenbaum,\emph{Values of zeta-functions, étale cohomology, and algebraic K-theory}, in: Algebraic K-theory, II: "Classical'' algebraic K-theory and connections with arithmetic (Proc. Conf., Battelle Memorial Inst., Seattle, Wash., 1972), Springer Lecture Notes in Math. 342, Springer, Berlin, 1973, 489--501.

\bibitem{MonodLifting}
N.~Monod, \emph{Equivariant measurable lifting}, { \tt arXiv:1312.1495 [math.FA]}.

\bibitem{MonodBook}
\bysame, \emph{Continuous bounded cohomology of locally compact groups},
  Lecture Notes in Mathematics, vol. 1758, Springer-Verlag, Berlin, 2001.

\bibitem{ModularForms}
W.~Nahm, \emph{Conformal field theory and torsion elements of the Bloch group}, in: Frontiers in number theory, physics, and geometry. II, Springer, Berlin, 2007, 67--132. 

\bibitem{CFT3}
W.~Nahm, A.~Recknagel and M.~Terhoeven, M., \emph{Dilogarithm identities in conformal field theory}
Modern Phys. Lett. A 8 (1993), no. 19, 1835 -- 1847.

\bibitem{Pieters}
H.~Pieters, \emph{Continuous cohomology of the isometry group of hyperbolic space realizable on the boundary}, arXiv preprint, 2015.

\bibitem{Rogers}
L.~J.~Rogers, \emph{On function sum theorems connected with the series $\sum_{n=1}^\infty x^n/n^2$}, Proc. London Math. Soc. 4, 1907, 169--189.

\bibitem{Schaeffer} W. Schaeffer, \emph{De integrali $-\int_0^x \frac{\log(1-\alpha)}{\alpha}d\alpha$}, J. Reine Angew Math. 30 (1846),
277 -- 295.

\bibitem{Ulam} S.~M.~Ulam, \emph{A collection of mathematical problems}, Interscience Tracts in Pure and Applied Mathematics, no. 8, Interscience Publishers, New York-London,
1960.

\bibitem{Zagier}
D.~Zagier, \emph{The dilogarithm function}, in: Frontiers in number theory, physics, and geometry. {II}, Springer, Berlin, 2007, 3--65. 

\end{thebibliography}

\end{document}